\documentclass[12pt]{amsart}
\usepackage{epsfig,color}

\makeatother

\theoremstyle{definition}

\def\fnum{equation}
\newtheorem{Thm}[\fnum]{Theorem}
\newtheorem{Cor}[\fnum]{Corollary}

\newtheorem{Lem}[\fnum]{Lemma}

\newtheorem{Def}[\fnum]{Definition}

\newtheorem{Rem}[\fnum]{Remark}
\newtheorem{Pro}[\fnum]{Proposition}

\numberwithin{equation}{section}

\newcommand{\Energy}{{\text{E}}}

\newcommand{\nn}{{\bf{n}}}
\newcommand{\Ric}{{\text{Ric}}}

\newcommand{\Tr}{{\text{Tr}}}

\newcommand{\dd}{{\text {d}}}

\def\K{{\text K}}

\def\RR{{\bold R}}
\def\RP{{\bold{RP}}}
\def\SS{{\bold S}}

\def\CC{{\bold C }}
\newcommand{\dv}{{\text {div}}}
\newcommand{\e}{{\text {e}}}

\newcommand{\Area}{{\text {Area}}}

\newcommand{\cC}{{\mathcal{C}}}

\newcommand{\cB}{{\mathcal{B}}}
\newcommand{\cG}{{\mathcal{G}}}

\newcommand{\cM}{{\mathcal{M}}}

\newcommand{\cP}{{\mathcal{P}}}
\newcommand{\cS}{{\mathcal{S}}}

\newcommand{\eqr}[1]{(\ref{#1})}

\begin{document}

\title[Width and finite  extinction time of Ricci flow]{Width and finite extinction time of Ricci flow}

\author{Tobias H. Colding}%
\address{MIT\\
77 Massachusetts Avenue, Cambridge, MA 02139-4307\\
and Courant Institute of Mathematical Sciences\\
251 Mercer Street, New York, NY 10012.}
\author{William P. Minicozzi II}%
\address{Department of Mathematics\\
Johns Hopkins University\\
3400 N. Charles St.\\
Baltimore, MD 21218}

\thanks{The   authors
were partially supported by NSF Grants DMS  0606629 and DMS
0405695}


\email{colding@math.mit.edu  and minicozz@math.jhu.edu}

\maketitle


\section{Introduction}

This is an expository article with complete proofs intended for a
general non-specialist audience.  The results are two-fold. First,
we discuss a geometric invariant, that we call the width, of a
manifold and show how it can be realized as the sum of areas of
minimal $2$-spheres.  For instance, when $M$ is a homotopy
$3$-sphere, the width is loosely speaking the area of the smallest
$2$-sphere needed to ``pull over'' $M$. Second, we use this to
conclude that Hamilton's Ricci flow becomes extinct in finite time
on any homotopy $3$-sphere. We have chosen to write this since the
results and ideas given here are quite useful and seem to be of
interest to a wide audience.

Given a Riemannian metric on a closed manifold $M$, sweep $M$ out
by a continuous one-parameter family of maps from $\SS^2$ to $M$
starting and ending at point maps. Pull the sweepout tight by, in
a continuous way, pulling each map as tight as possible yet
preserving the sweepout. We show the following useful property
(see Theorem \ref{t:existence} below); cf. 12.5 of \cite{Al},
proposition 3.1 of \cite{Pi}, proposition 3.1 of \cite{CD},
\cite{CM3},  and \cite{CM1}:

\vskip2mm
\parbox{6in}{Each map in the tightened
sweepout whose area is close to the width (i.e., the maximal
energy of the maps in the sweepout) must itself be close to a
collection of harmonic maps. In particular, there are maps in the
sweepout that are close to a collection of immersed minimal
$2$-spheres.}

\vskip2mm This useful property that {\emph{all}} almost maximal
slices are close to critical points is virtually always implicit
in any sweepout construction of critical points for variational
problems yet it is not always recorded since most authors are only
interested in  existence of a critical point.

Similar results hold for sweepouts by curves{\footnote{ Finding
closed geodesics on the $2$-sphere by using sweepouts goes back to
Birkhoff in 1917; see \cite{B1}, \cite{B2},
 section $2$ in \cite{Cr}, and \cite{CM3}. In the 1980s
Sacks-Uhlenbeck, \cite{SaU}, found minimal $2$-spheres  on general
manifolds using Morse theoretic arguments that are essentially
equivalent to sweepouts; a few years later, Jost explicitly used
sweepouts to obtain minimal $2$-spheres in \cite{Jo}. The argument
given here works equally well on any closed manifold, but only
produces non-trivial minimal objects when the width is positive.}}
instead of $2$-spheres; cf. \cite{CM3} where sweepouts by curves
are used to estimate the rate of change of a $1$-dimensional width
for convex hypersurfaces in Euclidean space flowing by positive
powers of their mean curvatures.  The ideas are essentially the
same whether one sweeps out by curves or $2$-spheres, though the
techniques in the curve case are purely ad hoc whereas for
sweepouts by $2$-spheres additional techniques, developed in the
1980s, have to be used to deal with energy concentration (i.e.,
``bubbling''); cf. \cite{SaU} and \cite{Jo}. The basic idea in
each of the two cases is a local replacement process that can be
thought of as a discrete gradient flow.  For curves, this is now
known as Birkhoff's curve shortening process; see \cite{B1},
\cite{B2}.

Local replacement had already been used by H.A. Schwarz in 1870 to
solve the Dirichlet problem in general domains, writing the domain
as a union of overlapping balls, and using that a solution can be
found explicitly on balls by, e.g., the Poisson formula; see
\cite{Sc1} and \cite{Sc2}. His method, which is now known as
Schwarz's alternating method, continues to play an important role
in applied mathematics, in part because the replacements converge
rapidly to the solution. The underlying reason why both Birkhoff's
method of finding closed geodesics and Schwarz's method of solving
the Dirichlet problem converge is convexity.  We will deviate
slightly from the usual local replacement argument and prove a new
convexity result for harmonic maps.  This allows us to make
replacements on balls with small energy, as opposed to balls with
small $C^0$ oscillation.   It is, in our view, much more natural
to make the replacement based on energy and gives, as a
bi-product, a new uniqueness theorem for harmonic maps since
already in dimension two the Sobolev embedding fails to control
the $C^0$ norm in terms of the energy; see Figure \ref{f:f1}.

The second thing we do is explain how to use this property of the
width  to show that on a homotopy $3$-sphere, or more generally
closed $3$-manifolds without aspherical summands, the Ricci flow
becomes extinct in finite time. This was shown by Perelman in
\cite{Pe} and by Colding-Minicozzi in \cite{CM1}; see also
\cite{Pe} for applications to the elliptic part of geometrization.

\begin{figure}[htbp]
 \setlength{\captionindent}{10pt}
    \begin{minipage}[t]{0.5\textwidth}
    \centering\includegraphics[width=2in]{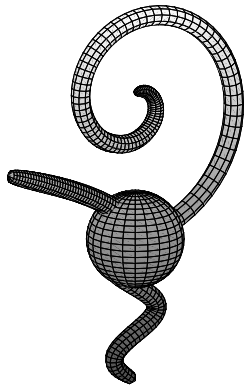}
    \caption{A conformal map to a long thin surface with small area has little energy.
In fact, for a conformal map, the part of the map that goes
 to small area tentacles  contributes little energy and
will be truncated by harmonic replacement.}
    \label{f:f1}
    \end{minipage}\begin{minipage}[t]{0.5\textwidth}
    \centering\input{CMjamsfig.pstex_t}
    \caption{The sweepout, the min--max surface, and the width W.}
    \label{f:f2}
    \end{minipage}
\end{figure}

We would like to thank Fr\'ed\'eric H\'elein, Bruce Kleiner, and
John Lott for their comments.

\section{Width and finite extinction}

On a homotopy $3$-sphere there is a natural way of constructing
minimal surfaces and that comes from the min-max argument where
the minimal of all maximal slices of sweepouts is a minimal
surface.  In \cite{CM1} we looked at how the area of this min-max
surface changes under the flow.  Geometrically the area measures a
kind of width of the $3$-manifold (see Figure \ref{f:f2}) and for
$3$-manifolds without aspherical summands (like a homotopy
$3$-sphere) when the metric evolve by the Ricci flow, the area
becomes zero in finite time corresponding to that the solution
becomes extinct in finite time.\footnote{It may be of interest to
compare our notion of width, and the use of it, to a well-known
approach to the Poincar\'e conjecture.  This approach asks to show
that for any metric on a homotopy $3$-sphere a min-max type
argument produces an \underline{embedded} minimal $2$-sphere. Note
that in the definition of the width it play no role whether the
minimal $2$-sphere is embedded or just immersed, and thus, the
analysis involved in this was settled a long time ago. This
well-known approach has been considered by many people, including
Freedman, Meeks, Pitts, Rubinstein, Schoen, Simon, Smith, and Yau;
see \cite{CD}.}

\subsection{Width}  \label{ss:width}

Let $\Omega$ be the set of continuous maps $\sigma : \SS^2 \times
[0,1] \to M$ so that
 for each $t\in [0,1]$ the map $\sigma (\cdot , t )$
is in $C^0 \cap W^{1,2}$, the map  $t \to \sigma (\cdot , t )$ is
continuous
 from $[0,1]$ to
$C^0 \cap W^{1,2}$, and finally $\sigma$ maps $\SS^2 \times \{ 0
\}$ and $\SS^1 \times \{ 1 \}$ to points.   Given a map
    $\beta \in \Omega$, the homotopy class $\Omega_{\beta}$
 is defined to be the set of maps $\sigma \in \Omega$ that are homotopic to
    $\beta$ through maps in $\Omega$.    We will call any such
     $\beta$  a
{\emph{sweepout}}; some authors use a more restrictive notion
where $\beta$ must also induce a degree one map from $\SS^3$ to
$M$.
  We will, in fact, be most interested in the case where
$\beta$ induces a map from $\SS^3$ to $M$ in a non-trivial
class{\footnote{For example, when $M$ is a homotopy $3$-sphere and
the induced map has degree one.}} in $\pi_3 (M)$.  The reason for
this is that the width is positive in this case and, as we will
see, equal to the area of a non-empty collection of minimal
$2$-spheres.

The (energy) width $W_E = W_E (\beta , M)$ associated to the homotopy class
 $\Omega_{\beta}$ is defined by taking the infimum of the maximum of
the energy of each slice.  That is,  set
\begin{equation}    \label{e:width}
    W_E = \inf_{  \sigma \in \Omega_{\beta}  } \,
       \,  \max_{ t \in [0, 1]} \,  \Energy \, (\sigma (\cdot , t ))
          \, ,
\end{equation}
where the energy is given by
\begin{equation}
    \Energy \, (\sigma
(\cdot , t )) = \frac{1}{2} \, \int_{\SS^2} \, \left| \nabla_x
\sigma (x,t) \right|^2 \, dx \, .
\end{equation}
 Even though this type of
construction is always called min-max, it is really
inf-max.  That is, for each (smooth) sweepout one looks
at the maximal energy of the slices and then takes the infimum
over all sweepouts in a given homotopy class.  The width is always
non-negative by definition, and positive when the homotopy class
of $\beta$ is non-trivial.
  Positivity can, for instance, be seen directly using \cite{Jo}.  Namely,
page 125 in \cite{Jo} shows that if $\max_t \Energy (\sigma (\cdot
, t ))$ is sufficiently small (depending on $M$), then $\sigma$ is
homotopically trivial.{\footnote{See the remarks after Corollary
\ref{c:trivmap2} for a different proof.}}

   One could alternatively define the width
using area rather than energy by setting
\begin{equation}    \label{e:widtharea}
    W_A = \inf_{  \sigma \in \Omega_{\beta}  } \,
       \,  \max_{ t \in [0, 1]} \,  \Area \, (\sigma (\cdot , t ))
          \, .
\end{equation}
The area of a $W^{1,2}$ map $u: \SS^2 \to \RR^N$ is by definition
the integral of the Jacobian $J_u = \sqrt{ \det \left( du^T \, du
\right) }$, where $du$ is the differential of $u$ and $du^T$ is
its transpose.
 That is, if $e_1 , e_2$ is an orthonormal frame on
  $D \subset \SS^2$, then  $J_u = \left( |  u_{e_1} |^2
\,  |  u_{e_2} |^2 - \langle u_{e_1} , u_{e_2}
        \rangle^2 \right)^{\frac{1}{2}} \leq \frac{1}{2} \, |du|^2$ and
\begin{equation}        \label{e:tr1}
         \Area (u \big|_D)= \int_{D} J_u \leq \Energy (u\big|_D)
                   \, .
\end{equation}
Consequently, area is less than or equal to energy with equality
if and only if $\langle  u_{e_1} , u_{e_2}
        \rangle$ and $|u_{e_1}|^2 - |u_{e_2}|^2$ are zero (as $L^1$ functions).
In the case of equality, we say that $u$ is {\it almost
conformal}.
 As in the classical Plateau problem (cf. Section 4 of
\cite{CM2}), energy is somewhat easier to work with  in proving
the existence of minimal surfaces. The next proposition, proven in
Appendix \ref{s:eq}, shows that $W_E=W_A$ as for the Plateau
problem (clearly, $W_A\leq W_E$ by the discussion above).
Therefore, we will drop the subscript and just write $W$.

\begin{Pro}     \label{p:eq}
$W_E=W_A$.
\end{Pro}

\subsection{Finite extinction}

Let $M^3$ be a smooth closed orientable $3$-manifold and $g(t)$ a
one-parameter family of metrics on $M$ evolving by Hamilton's
Ricci flow, \cite{Ha1}, so
\begin{equation}  \label{e:eqRic}
 \partial_t g=-2\,\Ric_{M_t}\, .
\end{equation}
When $M$ is prime and non-aspherical, then  it follows by standard
topology that $\pi_3 (M)$ is non-trivial (see, e.g., \cite{CM1}).
For such an $M$,   fix a non-trivial homotopy class $\beta \in
\Omega$. It follows that the width $W(g(t)) = W(\beta ,
g(t))$ is positive for each metric $g(t)$. This positivity is the
only place where the assumption on the topology of $M$ is used in
 the theorem below giving an upper bound for the derivative of the
 width under the Ricci flow.  As a consequence, we
 get that the solution of the flow becomes extinct in finite time (see
paragraph 4.4 of \cite{Pe} for the precise definition of
extinction time when surgery occurs).

\begin{Thm}     \label{t:upper}
\cite{CM1}. Let $M^3$ be a closed orientable prime non-aspherical
$3$-manifold equipped with a   metric $g=g(0)$. Under the Ricci
flow, the width $W(g(t))$ satisfies
\begin{equation}   \label{e:di1a}
\frac{d}{dt} \, W(g(t))  \leq -4 \pi + \frac{3}{4 (t+C)} \,
W(g(t))   \, ,
\end{equation}
in the sense of the limsup of forward difference quotients. Hence,
 $g(t)$  becomes extinct in finite time.
\end{Thm}

The $4\pi$ in \eqr{e:di1a} comes from the Gauss-Bonnet theorem and
the $3/4$ comes from the bound on the minimum of the scalar
curvature that the evolution equation implies.  Both of these
constants matter whereas the constant $C > 0$ depends on the
initial metric and the actual value is not important.

To see that \eqr{e:di1a} implies finite extinction time rewrite
\eqr{e:di1a} as
\begin{equation}
\frac{d}{dt} \left( W(g(t)) \, (t+C)^{-3/4} \right) \leq - 4\pi \,
(t+C)^{-3/4}
\end{equation} and
integrate to get
\begin{equation}  \label{e:lastaa}
 (T+C)^{-3/4} \, W(g(T)) \leq C^{-3/4} \, W(g(0))
- 16 \, \pi \, \left[ (T+C)^{1/4} - C^{1/4} \right]   \, .
\end{equation}
Since $W \geq 0$ by definition and the right hand side of
\eqr{e:lastaa} would become negative for $T$ sufficiently large,
we get the claim.

\vskip2mm Theorem \ref{t:upper} shows, in particular, that the
Ricci flow becomes extinct for any homotopy $3$-sphere. In fact,
we get as a corollary finite extinction time for the Ricci flow on
all $3$-manifolds without aspherical summands (see $1.5$ of
\cite{Pe} or section $4$ of \cite{CM1} for why this easily
follows):

\begin{Cor}  \label{c:upper}
(\cite{CM1}, \cite{Pe}). Let $M^3$ be a closed orientable
$3$-manifold whose prime decomposition has only non-aspherical
factors and is equipped with a   metric $g=g(0)$. Under the Ricci
flow with surgery, $g(t)$   becomes extinct in finite time.
\end{Cor}

Part of Perelman's interest in the question about finite time
extinction  comes from the following:  If one is interested in
geometrization of a homotopy $3$-sphere (or, more generally, a
$3$-manifold without aspherical summands) and knew that the Ricci
flow became extinct in finite time, then one would not need to
analyze what happens to the flow as time goes to infinity. Thus,
in particular, one would not need collapsing arguments.

One of the key ingredients in the proof of Theorem \ref{t:upper}
is the existence of a  sequence of good  sweepouts of $M$, where
each map in the sweepout whose area is close to the width (i.e.,
the maximal energy of any map in the sweepout) must itself be
close to a collection of harmonic maps.  This will be given by
Theorem \ref{t:existence} below, but we will first need a notion
of closeness and a notion of convergence of maps from $\SS^2$ into
a manifold.

\subsection{Varifold convergence} \label{ss:var}
Fix a closed manifold $M$ and let
$\Pi: G_kM\rightarrow  M$
     be the Grassmanian bundle of
(un-oriented) $k$-planes, that is, each fiber $\Pi^{-1}(p)$ is the
set of all $k$-dimensional linear subspaces of the tangent space
of $M$ at $p$.  Since $G_kM$ is compact, we can choose a countable
dense subset $\{h_n\}$ of all continuous functions on $G_kM$ with
supremum norm at most one (dense with respect to the supremum
norm).{\footnote{This is a corollary of the Stone-Weierstrass
theorem; see corollary $35$ on page $213$ of \cite{R}.}} If
$(X_0,F_0)$ and $(X_1, F_1)$ are two compact (not necessarily
connected) surfaces $X_0$, $X_1$ with measurable maps $F_i:X_i\to
G_kM$ so that each $f_i=\Pi \circ F_i$ is in $W^{1,2}(X_i ,M)$ and
$J_{f_i}$ is the Jacobian of $f_i$, then the varifold distance
between them is by definition
\begin{equation}
\dd_V (F_0,F_1)=\sum_n 2^{-n} \left|\int_{X_0} h_n\circ F_0\,
J_{f_0} -\int_{X_1} h_n\circ F_1\,  J_{f_1} \right|\, .
\end{equation}
It follows easily that a sequence $X_i=(X_i, F_i)$ with uniformly
bounded areas converges to $(X,F)$, iff it converges weakly, that
is, if for all $h\in C^0(G_2M)$ we have $\int_{X_i}h\circ F_i\,
J_{f_i} \to \int_{X} h\circ F\, J_f$. For instance, when $M$ is a
$3$-manifold, then $G_2M$, $G_1M$, and $T^1M/\{\pm v\}$ are
isomorphic.  (Here $T^1M$ is the unit tangent bundle.)
If $\Sigma_i$ is a sequence of closed immersed surfaces in $M$
converging to a closed surface $\Sigma$ in the usual $C^k$
topology, then we can think of each surface as being embedded in
$T^1M/\{\pm v\}\equiv G_2M$ by mapping each point to plus-minus the unit
normal vector, $\pm \nn$, to the surface.  It follows easily that
the surfaces with these inclusion maps converges in the varifold
distance.  More generally, if $X$ is a compact surface and $f:X\to
M$ is a $W^{1,2}$ map, where $M$ is no longer assumed to be
$3$-dimensional, then we let $F:X\to G_2M$ be given by that $F(x)$
is the linear subspace $df(T_xX)$.  (When $M$ is $3$-dimensional,
then we may think of the image of this map as lying in $T^1M/\{\pm
v\}$.) Strictly speaking, this is only defined on the measurable
space, where $J_f$ is non-zero; we extend it arbitrarily
to all of $X$ since the corresponding Radon measure on $G_2M$ given by
  $h\to \int_{X} h\circ F\, J_f$ is independent of the
extension.

\subsection{Existence of good sweepouts}

A $W^{1,2}$ map $u$ on a smooth compact surface $D$ with boundary
$\partial D$ is    {\emph{energy minimizing}} to $M \subset \RR^N$
if $u(x)$ is in $M$ for almost every $x$ and
\begin{equation}
    \Energy (u) = \inf \, \{ \Energy (w) \, | \, w\in W^{1,2}(D,M)
\text{ and }\, (w-u) \in
    W_0^{1,2}(D) \} \, .
\end{equation}
The map $u$ is said to be {\emph{weakly harmonic}} if $u$ is a
$W^{1,2}$ weak solution of the harmonic map equation $\Delta u
\perp TM$; see, e.g., lemma $1.4.10$ in \cite{He1}.

The next result gives the existence of a sequence of good
sweepouts.

\begin{Thm}     \label{t:existence}
Given a metric $g$ on $M$ and a map $\beta  \in \Omega$
representing a non-trivial class in $\pi_3 (M)$, there exists  a
sequence of sweepouts $\gamma^j \in \Omega_{\beta}$ with
$\max_{s \in [0,1]} \, \Energy (\gamma^j_s)\to W(g)$,
and so that given $\epsilon >
 0$, there exist $\bar{j}$ and $\delta > 0$
so that if $j > \bar{j}$ and
\begin{equation}    \label{e:eclose}
      \Area (\gamma^j (\cdot , s)) > W(g) - \delta \, ,
\end{equation}
then there are finitely many harmonic maps $u_i : \SS^2 \to M$
 with
\begin{equation}    \label{e:bclose}
    \dd_V  \, (\gamma^j (\cdot , s) , \cup_i \{ u_i \}
     ) < \epsilon \, .
\end{equation}
\end{Thm}

\vskip2mm One immediate consequence of Theorem \ref{t:existence}
is that if $s_j$ is any sequence with $\Area (\gamma^j (\cdot ,
s_j))$ converging to the width $W(g)$ as $j \to \infty$, then a
subsequence of
 $\gamma^j (\cdot ,
s_j)$ converges to a collection of harmonic maps from $\SS^2$ to
$M$.  In particular, the sum of the areas of these maps is exactly
$W(g)$ and, since the maps are automatically conformal, the sum of
the energies is also $W(g)$. The existence of at least one
non-trivial harmonic map from $\SS^2$ to $M$ was first proven in
\cite{SaU}, but they allowed for loss of energy in the limit; cf.
also \cite{St}. This energy loss was ruled out by Siu and Yau,
using also arguments of Meeks and Yau (see Chapter VIII in
\cite{SY}). This was also proven later by Jost in theorem $4.2.1$
of \cite{Jo} which  gives at least one min-max sequence converging
to a collection of harmonic maps.  The convergence in \cite{Jo} is
in a different topology that, as we will see in Appendix
\ref{a:A}, implies varifold convergence.


\subsection{Upper bounds for the rate of change of width}

Throughout this subsection, let $M^3$ be a smooth closed prime and
non-aspherical orientable $3$-manifold and let $g(t)$ be a
one-parameter family of metrics on $M$ evolving by the Ricci flow.
We will prove Theorem \ref{t:upper} giving the upper bound for the
derivative of the width $W(g(t))$ under the Ricci flow. To do
this, we need three things.

One is that the evolution equation for the scalar curvature
   $R= R(t)$,  see   page 16 of \cite{Ha2},
\begin{equation}    \label{e:prescalar}
    \partial_t R = \Delta R + 2 |\Ric|^2   \geq \Delta R +
    \frac{2}{3} \, R^2 \, ,
\end{equation}
implies by a straightforward maximum principle argument that at
time $t > 0$
\begin{equation}    \label{e:scalar}
   R(t) \geq  \frac{1}{1/[ \min R(0)] -2t/3} =  -\frac{3}{2 (t+C)} \, .
\end{equation}
The curvature is normalized so that on the unit $\SS^3$ the Ricci
curvature is $2$ and the scalar curvature is $6$. In the
derivation of \eqr{e:scalar} we implicitly assumed that $\min
R(0)<0$.  If this was not the case, then \eqr{e:scalar} trivially
holds for any $C>0$, since, by \eqr{e:prescalar}, $\min R (t)$ is
always non-decreasing.  This last remark is also used when surgery
occurs.  This is because by construction any surgery region has
large (positive) scalar curvature.

The second thing that we need in the proof is the observation that
if $\{ \Sigma_i \}$  is a collection of branched minimal
$2$-spheres and $f\in W^{1,2}(\SS^2,M)$ with $\dd_V  \, (f ,
\cup_i \Sigma_i ) < \epsilon$, then for any smooth quadratic form
$Q$  on $M$  we have (the unit normal $\nn_f$ is defined where
$J_f \ne 0$)
\begin{equation}     \label{e:hclose}
   \left| \int_{ f } [\Tr (Q) - Q(\nn_{f} ,
   \nn_{f})] - \sum_i \int_{\Sigma_i}
   [\Tr (Q) - Q(\nn_{\Sigma_i} ,
   \nn_{\Sigma_i})]  \right| <   C \, \epsilon \, \| Q \|_{C^1}
\, \Area (f) \, .
\end{equation}

The last thing   is an upper bound for the rate of change of area
of minimal $2$-spheres.
 Suppose that
$X$ is a closed surface and $f:X\to M$ is a
$W^{1,2}$ map, then using \eqr{e:eqRic}
an easy calculation gives (cf. pages 38--41 of \cite{Ha2})
\begin{equation}        \label{e:diffAn}
\frac{d}{dt}_{t=0}\Area_{g(t)}( f ) =-\int_{f} [R - \Ric_M
(\nn_f,\nn_f)] \, .
\end{equation}
If $\Sigma\subset M$ is a closed immersed minimal surface, then
\begin{equation}        \label{e:diffA}
\frac{d}{dt}_{t=0}\Area_{g(t)}( \Sigma )
   =-\int_{\Sigma} \K_{\Sigma}-\frac{1}{2}\int_{\Sigma}[ |A|^2 + R] \,
.
\end{equation}
Here $\K_{\Sigma}$ is the (intrinsic) curvature of $\Sigma$,  $A$
is the second fundamental form of $\Sigma$, and $|A|^2$ is the sum
of the squares of the principal curvatures.  To get \eqr{e:diffA}
from \eqr{e:diffAn}, we used that if $\K_M$ is the sectional
curvature of $M$ on the two-plane tangent to $\Sigma$, then the
Gauss equations and minimality of $\Sigma$ give
$\K_{\Sigma}=\K_M-\frac{1}{2}|A|^2$. The next lemma gives the
upper bound.

\begin{Lem}     \label{l:upper}
If $\Sigma\subset M^3$ is a branched minimal immersion of the
$2$-sphere, then
\begin{equation}    \label{e:in24}
\frac{d}{dt}_{t=0}\Area_{g(t)}(\Sigma)  \leq -4 \pi -
\frac{\Area_{g(0)}(\Sigma)}{2} \, \min_{M} R(0) \, .
\end{equation}
\end{Lem}

\begin{proof}
Let $\{ p_i \}$ be the set of branch points of $\Sigma$ and $b_i >
0$ the order of branching.
 By \eqr{e:diffA}
\begin{equation}
 \frac{d}{dt}_{t=0} \Area_{g(t)}( \Sigma )
\leq -\int_{\Sigma} \K_{\Sigma}-\frac{1}{2}\int_{\Sigma}R = -4\pi
-2\pi
 \sum b_i -\frac{1}{2}\int_{\Sigma}R  \, ,
\end{equation}
where the  equality used the Gauss-Bonnet theorem  with branch
points (this equality also follows from the Bochner type formula
for harmonic maps between surfaces given on page $10$ of \cite{SY}
and the second displayed equation on page $12$ of \cite{SY} that
accounts for the branch points). Note that branch points only help
in the inequality \eqr{e:in24}.
\end{proof}

Using these three things, we can show the upper bound for the rate
of change of the width.

\vskip2mm
\begin{proof}
(of Theorem \ref{t:upper})   Fix a time $\tau$.  Below $\tilde{C}$
denotes a constant depending only on $\tau$ but will be allowed to
change from line to line.  Let $\gamma^j (\tau)$ be the sequence
of sweepouts for the metric $g(\tau)$ given by
 Theorem \ref{t:existence}. We will use the sweepout at
 time $\tau$ as a comparison to get an upper bound for the width
 at times $t > \tau$.   The key for this is the following claim:
Given $\epsilon > 0$, there exist $\bar{j}$ and $\bar{h} > 0$ so
that if $j> \bar{j}$ and $0 < h < \bar{h}$, then
\begin{align}        \label{e:acomp1}
    \Area_{g(\tau + h)}( \gamma^j_s (\tau ) )
    &- \max_{s_0} \, \Area_{g(\tau)}( \gamma^j_{s_0} (\tau ))  \notag
    \\
&\leq
     [-4 \pi + \tilde{C} \, \epsilon +
 \frac{3}{4 (\tau+C)} \, \max_{s_0} \,
\Area_{g(\tau)}( \gamma^j_{s_0} (\tau )) ] \, h + \tilde{C} \, h^2
\, .
\end{align}
 To see why
\eqr{e:acomp1} implies \eqr{e:di1a}, use the equivalence of the
two definitions of widths to get
\begin{equation}        \label{e:defw}
    W(g (\tau + h) ) \leq \max_{s \in [0,1]}
    \Area_{g(\tau + h)}( \gamma^j_s (\tau ) ) \, ,
\end{equation}
and take the limit as $j\to \infty$ (so that{\footnote{This
follows by combining that $\Area_{g(\tau)}( \gamma^j_{s_0} (\tau
)) \leq \Energy_{g(\tau)}( \gamma^j_{s_0} (\tau ))$ by
\eqr{e:tr1}, $\max_{s_0} \, \Energy_{g(\tau)}( \gamma^j_{s_0}
(\tau ))\to W(g(\tau))$, and $W(g(\tau)) \leq \max_{s_0} \,
\Area_{g(\tau)}( \gamma^j_{s_0} (\tau ))$ by the equivalence of
the two definitions of width.}}
 $\max_{s_0} \, \Area_{g(\tau)}( \gamma^j_{s_0} (\tau ))\to
W(g(\tau))$) in \eqr{e:acomp1} to get
\begin{equation}        \label{e:defwq}
    \frac{W(g (\tau + h) ) - W(g (\tau ))}{h}
\leq    -4 \pi + \tilde{C} \, \epsilon +
 \frac{3}{4 (\tau+C)} \, W(g(\tau))     + \tilde{C} \, h
     \, .
\end{equation}
Taking $\epsilon \to 0$ in \eqr{e:defwq} gives \eqr{e:di1a}.

  It remains to prove
\eqr{e:acomp1}. First,  let $\delta > 0$ and $\bar{j}$, depending
on $\epsilon$ (and on $\tau$), be given by Theorem
\ref{t:existence}. If $j
> \bar{j}$ and $\Area_{g(\tau)} (\gamma^j_s (\tau )  ) >  W(g) -
\delta$, then let $\cup_i \Sigma_{s,i}^j (\tau) $ be the
collection of minimal spheres given by Theorem \ref{t:existence}.
Combining \eqr{e:diffAn}, \eqr{e:hclose} with $Q = \Ric_M$, and
Lemma \ref{l:upper} gives
\begin{align}        \label{e:diffAn2}
    \frac{d}{dt}_{t=\tau}\Area_{g(t)}( \gamma^j_s (\tau )) &
    \leq \frac{d}{dt}_{t=\tau}\Area_{g(t)}( \cup_i \Sigma_{s,i}^j (\tau) )
    +  \tilde{C} \, \epsilon \, \| \Ric_M \|_{C^1}
\, \Area_{g(\tau)}( \gamma^j_s (\tau ))\notag \\
        &\leq -4 \pi  - \frac{
\Area_{g(\tau)}( \gamma^j_s (\tau ))}{2} \, \min_{M} R(\tau) +
\tilde{C} \, \epsilon \\
&\leq -4 \pi  +
 \frac{3}{4 (\tau+C)} \, \max_{s_0} \,
\Area_{g(\tau)}( \gamma^j_{s_0} (\tau ))  + \tilde{C} \, \epsilon
\notag \, ,
\end{align}
where the last inequality used the lower bound \eqr{e:scalar} for
$R(\tau)$. Since the metrics $g(t)$ vary smoothly and every
sweepout $\gamma^j$ has uniformly bounded energy, it is easy to
see that $\Area_{g(\tau + h)} (\gamma^j_s (\tau )  )$ is a smooth
function of $h$ with a uniform $C^2$ bound independent of both $j$
and $s$ near $h=0$ (cf. \eqr{e:diffAn}).  In particular,
\eqr{e:diffAn2} and Taylor expansion give $\bar{h} > 0$
 (independent of $j$) so that \eqr{e:acomp1} holds
for  $s$ with $\Area_{g(\tau)} (\gamma^j_s (\tau )  ) >  W(g) -
\delta$.  In
 the remaining case, we have  $\Area (\gamma^j_s (\tau )) \leq
W(g) - \delta$ so the continuity of $g(t)$ implies that
\eqr{e:acomp1} automatically holds after possibly shrinking
$\bar{h}> 0$.
\end{proof}

\subsection{Parameter spaces}
Instead of using the unit interval, $[0,1]$, as the parameter
space for the maps in the sweepout and assuming that the maps
start and end in point maps, we could have used any compact
finite dimensional topological space $\cP$ and required that the
maps are constant on $\partial \cP$ (or that $\partial \cP =
\emptyset$). In this case, let $\Omega^{\cP}$ be the set of
continuous maps $\sigma : \SS^2 \times \cP \to M$ so that
 for each $t \in \cP$ the map $\sigma (\cdot , t )$
is in $C^0 \cap W^{1,2}(\SS^2 , M)$, the map  $t \to \sigma (\cdot
, t )$ is continuous
 from $\cP$ to
$C^0 \cap W^{1,2}(\SS^2 , M)$, and finally $\sigma$ maps $\partial
\cP$  to point maps. Given a map
    $\hat{\sigma} \in \Omega^{\cP}$, the homotopy class
$\Omega^{\cP}_{\hat{\sigma}} \subset \Omega^{\cP}$  is defined to
be the set of maps $\sigma \in \Omega^{\cP}$ that are homotopic to
    $\hat{\sigma}$ through maps in $\Omega^{\cP}$.  Finally,
    the
    width $W=W(\hat{\sigma})$ is
$\inf_{  \sigma \in \Omega^{\cP}_{\hat{\sigma}}  } \,
       \,  \max_{ t \in \cP} \,  \Energy \, (\sigma (\cdot , t ))
        $.
With only trivial changes, the same proof yields Theorem
\ref{t:existence}  for these general parameter
spaces.{\footnote{The main change is in Lemma \ref{l:goodballs}
below where the bound $2$ for the multiplicity in (1) becomes
$\dim (\cP) + 1$.  This follows from the definition of (covering)
dimension; see pages $302$ and $303$ in \cite{Mu}.}}

\section{The energy decreasing map and its consequences}

To prove Theorem \ref{t:existence}, we will first define an energy
decreasing map from $\Omega$ to itself that preserves the homotopy
class (i.e., maps each $\Omega_{\beta}$ to itself)
 and record its key properties.  This should be thought of as a
generalization of Birkhoff's curve shortening process that plays a
similar  role   when tightening a sweepout by curves; see
\cite{B1}, \cite{B2}, \cite{Cr}, and \cite{CM3}.

Throughout this paper,  by a {\emph{ball}} $B \subset \SS^2$, we
will mean a subset of $\SS^2$ and a stereographic projection
$\Pi_B$ so that $\Pi_B (B) \subset \RR^2$ is a ball.
 Given $\rho
> 0$, we will let $\rho \, B \subset \SS^2$ denote $\Pi_B^{-1}$
of the ball with the same center as $\Pi_B (B)$ and radius $\rho$
times that of $\Pi_B (B)$.

\begin{Thm}     \label{p:tilde}
There is a constant $\epsilon_0 > 0$ and a continuous function
$\Psi:[0, \infty) \to [0,\infty)$ with $\Psi (0)=0$, both
depending on $M$, so that given any $\tilde{\gamma} \in \Omega$
without non-constant harmonic slices and $W>0$, there exists
${\gamma} \in \Omega_{\tilde{\gamma}}$ so that $\Energy (\gamma
(\cdot , t)) \leq \Energy (\tilde{\gamma}(\cdot , t))$ for each
$t$ and so for each $t$ with $\Energy (\tilde{\gamma} (\cdot , t))
\geq W/2$:
\begin{enumerate}
\item[($B_{\Psi}$)] If $\cB$ is any finite collection
of disjoint closed balls in $\SS^2$  with $\int_{\cup_{\cB} B}
|\nabla \gamma (\cdot , t)|^2 < \epsilon_0$ and $v: \cup_{\cB}
\frac{1}{8} B \to M$ is an energy minimizing map equal to
${\gamma} (\cdot , t)$ on $ \cup_{\cB} \frac{1}{8} \partial \, B$,
then
\begin{equation}    \label{e:condtilde}
    \int_{\cup_{\cB} \frac{1}{8}
B} \left| \nabla \gamma (\cdot , t) - \nabla v \right|^2 \leq \Psi
    \left[
\Energy (\tilde{\gamma} (\cdot , t)) - \Energy ( {\gamma} (\cdot ,
t))\right]  \, .  \notag
\end{equation}
\end{enumerate}
\end{Thm}

The proof of Theorem \ref{p:tilde} is given in Section
\ref{s:birk}.  The second ingredient that we will need to prove
  Theorem
\ref{t:existence}  is a compactness result that generalizes
compactness of harmonic maps to maps that are closer and closer to
being harmonic (this is Proposition \ref{p:gl2} below and will be
proven in Appendix \ref{s:ppp}).

\subsection{Compactness of almost harmonic maps}

 Our notion of almost harmonic relies on two
important properties of harmonic maps from $\SS^2$ to $M$. The
first is that harmonic maps from $\SS^2$ are conformal and, thus,
energy and area are equal; see (A) below.  The second is that any
harmonic map from a surface is energy minimizing when restricted
to balls where the energy is sufficiently small; see (B) below.

In the proposition,  $\epsilon_{SU} >0$ (depending on $M$) is the
small energy constant from  lemma
 $3.4$ in  \cite{SaU}, so that we get interior estimates for
 harmonic maps with energy at most $\epsilon_{SU}$.  In
 particular, any non-constant harmonic map from $\SS^2$ to $M$
 has energy greater than $\epsilon_{SU}$.

 \begin{Pro}     \label{p:gl2}
Suppose that   $\epsilon_0 , E_0 > 0$ are constants with
$\epsilon_{SU} > \epsilon_0$ and
 $u^j:\SS^2 \to M$ is a sequence of  $C^0 \cap W^{1,2}$ maps with
  $E_0 \geq \Energy(u^j)$ satisfying: \\ (A) $\Area (u^j) > \Energy (u^j) -
1/j$. \\ (B) For any finite collection $\cB$ of disjoint closed
balls in $ \SS^2$ with $\int_{\cup_{\cB} B}\, |\nabla u^j|^2 <
\epsilon_0$  there is an
 energy minimizing
map $v:   \cup_{\cB} \frac{1}{8}B \to M$ that equals $u^j$ on $
\cup_{\cB} \frac{1}{8} \partial B $ with
\begin{equation}
    \int_{\cup_{\cB} \frac{1}{8}B} \left| \nabla u^j - \nabla v \right|^2 \leq  1/j \,
    .  \notag
\end{equation}
If (A) and (B) are satisfied, then   a subsequence of the $u^j$'s
  varifold converges to a collection of harmonic maps $v^0 ,
\dots , v^m:\SS^2 \to M$.
\end{Pro}

One immediate consequence of Proposition \ref{p:gl2} is a
compactness theorem for sequences of {\underline{harmonic}} maps
with bounded energy. This was proven by Jost  in lemma $4.3.1$ in
\cite{Jo}. In fact, Parker proved compactness of bounded energy
harmonic maps in a stronger topology, with $C^0$ convergence in
addition to $W^{1,2}$ convergence; see theorem $2.2$ in \cite{Pa}.
Therefore, it is perhaps not surprising that a similar compactness
holds for sequences that are closer and closer to being harmonic
in the sense above.  However, it is useful to keep in mind that
Parker has constructed sequences of   maps where the Laplacian is
going to zero in $L^1$ and yet there is no convergent subsequence
(see proposition $4.2$ in \cite{Pa}).

Finally, we point out that Proposition  \ref{p:gl2}  can be
thought of as a discrete version of Palais-Smale Condition (C).
Namely, if we have a sequence of maps where the maximal energy
decrease from harmonic replacement goes to zero, then a
subsequence converges to a collection of harmonic maps.

\subsection{Constructing good sweepouts from the energy decreasing
map on $\Omega$} Given Theorem \ref{p:tilde} and Proposition
\ref{p:gl2}, we will prove Theorem \ref{t:existence}. Let
$\cG^{W+1}$ be the set of collections of harmonic maps from
$\SS^2$ to $M$ so that the sum of  the energies is at most $W+1$.

\begin{proof}
(of Theorem \ref{t:existence}.)
Choose a sequence of maps
$\tilde{\gamma}^j \in \Omega_{ {\beta}}$
 with
\begin{equation}    \label{e:minseq}
    \max_{t \in [0,1]} \, \, \Energy \, (\tilde{\gamma}^j (\cdot , t))
    < W + \frac{1}{j} \, ,
\end{equation}
and so that $\tilde{\gamma}^j (\cdot , t)$ is not harmonic unless
it is a constant map.{\footnote{To do this, first use Lemma
\ref{l:density} (density of $C^2$-sweepouts) to choose
$\tilde{\gamma}^j_1 \in \Omega_{ {\beta}}$  so  $t \to
\tilde{\gamma}^j_1 (\cdot , t)$ is continuous from $[0,1]$ to
$C^2$ and $\max_{t \in [0,1]} \, \, \Energy \, (\tilde{\gamma}^j_1
(\cdot , t))
    < W + \frac{1}{2j}$.  Using stereographic projection, we can
    view
$\tilde{\gamma}^j_1 (\cdot , t)$ as a map from $\RR^2$.
    Now fix a $j$.  The continuity in $C^2$ gives a uniform bound
     $\sup_{t \in [0,1]} \, \sup_{B_1}
|\nabla \tilde{\gamma}^j_1 (\cdot , t)|^2 \leq C$ for some $C$.
Choose $R >0$ with $4\pi C \, R^2 \leq 1/(2j)$.
     Define a map $\Phi: \RR^2 \to \RR^2$ in polar
    coordinates by: $\Phi (r, \theta) = (2r, \theta)$ for $r<
    R/2$, $\Phi (r, \theta) = (R, \theta)$ for $
    R/2 \leq r \leq R$, and $\Phi (r, \theta) = (r,\theta)$ for
    $R< r$.  Note that $\Phi$ is  homotopic to the
    identity, is conformal away from the annulus $B_R \setminus B_{R/2}$, and
    on $B_R \setminus B_{R/2}$ has $|\partial_r \Phi| =0$ and
    $|d\Phi| \leq 2$.  It follows that $\tilde{\gamma}^j ( \cdot , t) = \tilde{\gamma}^j_1
    (\cdot , t) \circ \Phi$ is in $\Omega_{\beta}$, satisfies
    \eqr{e:minseq}, and has $\partial_r \tilde{\gamma}^j ( \cdot , t) = 0$ on
     $B_{R}\setminus B_{R/2}$.  Since harmonic maps from
    $\SS^2$ are conformal (corollary $1.7$ in \cite{SaU}), any harmonic $\tilde{\gamma}^j ( \cdot ,
    t)$ is constant on $B_{R}\setminus B_{R/2}$ and, thus,
     constant on $\SS^2$ by unique continuation (theorem $1.1$ in \cite{Sj}).}}
We can assume that $W>0$ since otherwise $\Area (\tilde{\gamma}^j
( \cdot ,
    t)) \leq \Energy (\tilde{\gamma}^j ( \cdot ,
    t)) \to 0$ and
the theorem follows trivially.

 Applying Theorem \ref{p:tilde} to the
$\tilde{\gamma}^j$'s gives a
 sequence $\gamma^j \in \Omega_{\beta}$ where each
$ {\gamma}^j (\cdot , t)$ has energy at most that of
$\tilde{\gamma}^j (\cdot , t)$. We will argue by contradiction to
show that the $\gamma^j$'s have the desired property.  Suppose,
therefore, that there exist $j_k \to \infty$ and $s_k \in [0,1]$
with $\dd_V (\gamma^{j_k} ( \cdot , s_k), \cG^{W+1}) \geq \epsilon
> 0$ and $
 \Area (\gamma^{j_k} ( \cdot, s_k))
> W - 1/k$.
Thus, by
 \eqr{e:minseq} and the fact that $\Energy (\cdot) \geq \Area (\cdot) $, we get
 \begin{equation}    \label{e:ecloseq}
       \Energy (\tilde{\gamma}^{j_k} ( \cdot, s_k)) - \Energy (\gamma^{j_k} ( \cdot, s_k))   \leq
      \Energy (\tilde{\gamma}^{j_k} ( \cdot, s_k)) - \Area ({\gamma}^{j_k} ( \cdot, s_k))
      \leq 1/k + 1/j_k \to 0 \, ,
\end{equation}
and, similarly,  $\Energy \left( \gamma^{j_k} ( \cdot, s_k)
\right) - \Area  \left( \gamma^{j_k} ( \cdot, s_k) \right) \to 0$.
  Using \eqr{e:ecloseq} in Theorem
\ref{p:tilde} gives\\ (B) If $\cB$ is any collection of disjoint
closed balls in $ \SS^2$  with $\int_{\cup_{\cB} B}\, |\nabla
{\gamma}^{j_k} ( \cdot, s_k)|^2 < \epsilon_0$  and $v: \cup_{\cB}
\frac{1}{8}B \to M$ is an
 energy minimizing
map  that equals ${\gamma}^{j_k} ( \cdot, s_k)$ on $ \cup_{\cB}
\frac{1}{8} \partial B $, then
\begin{equation}    \label{e:lfv2a}
    \int_{ \cup_{\cB} \frac{1}{8}B } \left| \nabla {\gamma}^{j_k} ( \cdot, s_k) - \nabla v \right|^2 \leq  \Psi (1/k + 1/j_k) \,
     \to 0 \, .
\end{equation}
 Therefore, we can apply Proposition \ref{p:gl2} to get that a
subsequence  of the ${\gamma}^{j_k} ( \cdot, s_k)$'s varifold
converges to a collection of harmonic maps.
  However, this contradicts the lower
bound for the varifold distance to $\cG^{W+1}$, thus completing
the proof.
\end{proof}

\section{Constructing the energy decreasing map}    \label{s:birk}

\subsection{Harmonic replacement}    \label{s:caseB}

The energy decreasing map from $\Omega$ to itself will be given by
a repeated replacement procedure.  At each step, we replace a map
$u$ by a map $H(u)$ that coincides with $u$ outside a ball and
inside the ball is equal to an energy-minimizing map with the same
boundary values as $u$. This is often referred to as
{\emph{harmonic replacement}}.

 One of the key properties that makes harmonic replacement useful is
that the energy functional is strictly convex on small energy
maps.  Namely, Theorem \ref{l:trivmap} below  gives a uniform
lower bound for the gap in energy between a harmonic map   and a
$W^{1,2}$ map with the same boundary values; see  Appendix
\ref{s:appB} for the proof.

\begin{Thm}     \label{l:trivmap}
 There exists a constant $\epsilon_1 > 0$
(depending on $M$) so that if $u$ and $v$ are $W^{1,2}$ maps from
$B_1 \subset \RR^2$ to $M$, $u$ and $v$ agree on $\partial B_1$,
and $v$ is weakly harmonic with energy at most $\epsilon_1$, then
\begin{equation}    \label{e:trivmap}
    \int_{B_1} |\nabla u|^2 -  \int_{B_1} |\nabla v|^2 \geq
    \frac{1}{2} \,
        \int_{B_1} \left| \nabla v - \nabla u \right|^2  \, .
\end{equation}
\end{Thm}

An immediate corollary of Theorem \ref{l:trivmap} is uniqueness of
solutions to the Dirichlet problem for small energy maps (and also
that any such harmonic map minimizes energy).

\begin{Cor} \label{c:trivmap1}
Let  $\epsilon_1 > 0$  be as in Theorem \ref{l:trivmap}.  If $u_1$
and $u_2$ are $W^{1,2}$ weakly harmonic maps from $B_1 \subset
\RR^2$ to $M$, both with energy at most $\epsilon_1$, and they
agree on $\partial B_1$,
 then $u_1=u_2$.
\end{Cor}

\subsection{Continuity of harmonic replacement on $C^0
(\overline{B_1}) \cap W^{1,2}(B_1)$}

   The second consequence of Theorem \ref{l:trivmap} is that
harmonic replacement is continuous as a map from $C^0
(\overline{B_1}) \cap W^{1,2}(B_1)$ to itself if we restrict to
small energy maps.  (The norm on $C^0 (\overline{B_1}) \cap
W^{1,2}(B_1)$ is the sum of the sup norm and the $W^{1,2}$ norm.)

\begin{Cor} \label{c:trivmap2}
Let   $\epsilon_1 > 0$  be as in Theorem \ref{l:trivmap} and  set
\begin{equation}
    \cM = \{ u \in C^0(\overline{B_1}, M) \cap W^{1,2}(B_1, M) \, | \,
    \Energy (u)  \leq \epsilon_1 \}  \, .
\end{equation}
Given $u  \in \cM$, there is a unique energy minimizing map $w$
equal to $u$ on $\partial B_1$ and $w$ is in $\cM$.
 Furthermore,   there exists $C$ depending on $M$ so that if
$u_1 , u_2 \in \cM$ with corresponding energy minimizing maps $w_1
, w_2$, and we set $\Energy =  \Energy (u_1) + \Energy (u_2)$,
then
\begin{equation}    \label{e:trivmap2}
    \left| \Energy (w_1) - \Energy (w_2) \right| \leq C \,
    ||u_1 - u_2||_{C^0(\overline{B_1})} \,   \Energy   +  C \,
    ||\nabla u_1 - \nabla u_2||_{L^2(B_1)} \, \Energy^{1/2}   \, .
\end{equation}
Finally, the map from $u$ to $w$ is
 continuous as a map from $C^0
(\overline{B_1}) \cap W^{1,2}(B_1)$ to itself.
\end{Cor}

In the proof, we will use that since $M$ is smooth, compact and
embedded, there exists a $\delta
> 0$ so that for each $x$ in the $\delta$-tubular neighborhood $M_{\delta}$
of $M$ in $\RR^N$, there is a unique closest point $\Pi (x) \in M$
and  so the map $x \to \Pi(x)$ is smooth.
  $\Pi$ is called
{\emph{nearest point projection}}.  Furthermore, for any $x \in
M$, we have $|d\Pi_x (V)| \leq |V|$.  Therefore, there is a
constant $C_{\Pi}$ depending on $M$ so that  for any $x \in
M_{\delta}$, we have $|d\Pi_x (V)| \leq (1+ C_{\Pi}|x-\Pi(x)|)\,
|V|$.  In particular, we can choose $\hat{\delta} \in (0, \delta)$
so that $|d\Pi_x (V)|^2 \leq 2 \, |V|^2$ for any $x \in
M_{\hat{\delta}}$ and $V \in \RR^N$.

\begin{proof}
(of Corollary \ref{c:trivmap2}.)
 The existence of an energy
minimizing map $w \in W^{1,2}(B_1)$ was proven by Morrey in
\cite{Mo1}; by Corollary \ref{c:trivmap1}, $w$ is unique. The
continuity of $w$ on $\overline{B_1}$ is the main theorem of
\cite{Q}.{\footnote{Continuity also essentially follows from the
boundary regularity of Schoen and Uhlenbeck, \cite{SU2}, except
that \cite{SU2} assumes $C^{2,\alpha}$ regularity of the boundary
data.}}  It follows that $w \in \cM$.

\vskip1mm \noindent {\bf{Step 1: $\Energy (w)$ is uniformly
continuous}}. We can assume that $||u_1 -
u_2||_{C^0(\overline{B_1})} \leq \hat{\delta}$, since
\eqr{e:trivmap2} holds with $C = 1/ \hat{\delta}$ if $||u_1 -
u_2||_{C^0(\overline{B_1})} \geq \hat{\delta}$.   Define a map
$v_1$ by
\begin{equation}
    v_1 = \Pi \circ \left( w_2 + (u_1 - u_2) \right) \, ,
\end{equation}
so that $v_1$ maps to $M$ and agrees with $u_1$ on $\partial B_1$.
Using that $|d\Pi_x (V)| \leq |V|$ for $x\in M$ and $w_2$ maps to
$M$, we can estimate the energy of $v_1$ by
\begin{equation}    \label{e:ev1}
   \Energy ( v_1 )\leq (1 + C_{\Pi} ||u_1 - u_2||_{C^0(\overline{B_1})})^2 \,  \left[ \Energy (w_2) +
   2 (\Energy (w_2) \, \Energy (u_1 - u_2))^{1/2}
   + \Energy (u_1 - u_2) \right] \, ,
\end{equation}
where $C_{\Pi}$ is the Lipschitz norm of $d\Pi$ in
$M_{\hat{\delta}}$.  Since $v_1$ and $w_1$ agree on $\partial
B_1$, Corollary \ref{c:trivmap1} yields $\Energy (w_1) \leq
\Energy (v_1)$. By symmetry, we can  assume that $\Energy (w_2)
\leq \Energy (w_1)$ so that \eqr{e:ev1} implies \eqr{e:trivmap2}.

\vskip1mm \noindent {\bf{Step 2: The continuity of $u \to w$}}.
 Suppose  that $u,
u_j$ are in $\cM$ with $u_j\to u$ in $C^0(\overline{B_1}) \cap
W^{1,2}(B_1)$ and $w$ and $w_j$ are the corresponding energy
minimizing  maps.

We will first show that $w_j \to w$ in $W^{1,2}(B_1)$.  To do
this, set
\begin{equation}    \label{e:defvj}
    v_j = \Pi \circ \left( w + (u_j - u) \right) \, ,
\end{equation}
so that $v_j$ maps to $M$ and agrees with $u_j$ on $\partial B_1$.
Arguing as in \eqr{e:ev1} and using that $\Energy (w_j) \to
\Energy (w)$ by Step 1, we get that $[\Energy (v_j) - \Energy
(w_j)] \to 0$. Therefore, applying Theorem \ref{l:trivmap} to
     $w_j , \, v_j$ gives that $||w_j -
     v_j||_{W^{1,2}(B_1)} \to 0$.  Since $||u_j - u||_{C^0(\overline{B_1})\cap W^{1,2}(B_1)} \to
     0$  and $\Pi \circ w = w$,
 it follows that $||w -
     v_j||_{W^{1,2}(B_1)} \to 0$.  The triangle inequality  gives
       $||w -
     w_j||_{W^{1,2}(B_1)} \to 0$.

Finally, we will argue by contradiction to see that $w_j \to w$ in
$ C^0(\overline{B_1})$.  Suppose instead that there is a
subsequence (still denoted $w_j$) with
\begin{equation}  \label{e:notc0}
   ||w_j -
w||_{C^0(\overline{B_1})} \geq \epsilon > 0 \, .
\end{equation}
 Using the uniform energy bound for the $w_j$'s together with
interior estimates for energy minimizing maps of \cite{SU1} (and
the Arzela-Ascoli theorem), we can pass to a further subsequence
so that  the $w_j$'s converge uniformly in $C^2$ on any compact
subset $K \subset B_1$. Finally, as remarked in the proof of the
main theorem in \cite{Q}, proposition $1$ and remark $1$ of
\cite{Q} imply that the $w_j$'s are also equicontinuous near
  $\partial B_1$,
  so Arzela-Ascoli gives a further subsequence that   converges
uniformly on $\overline{B_1}$ to a harmonic map $w_{\infty}$ that
agrees with $w$ on the boundary.  However, \eqr{e:notc0} implies
that $||w- w_{\infty}||_{C^0(\overline{B_1})} \geq \epsilon > 0$
which contradicts the uniqueness of small energy harmonic maps.
This completes the proof.
\end{proof}

Corollary \ref{c:trivmap2} gives another proof that the width is
  positive when the homotopy
class is non-trivial or, equivalently, that if
   $\max_t \Energy (\sigma (\cdot
, t ))$ is sufficiently small (depending on $M$), then $\sigma$ is
homotopically trivial.  Namely, since $t \to \sigma (\cdot , t)$
is continuous from $[0,1]$ to $C^0$, we can choose $r>0$ so that
$\sigma (\cdot , t)$ maps the ball $B_r (p) \subset \SS^2$ into a
convex geodesic ball $B^t$ in $M$ for every $t$. If each $\sigma
(\cdot , t)$ has energy less than $\epsilon_1 > 0$ given by
Corollary \ref{c:trivmap2}, then replacing $\sigma (\cdot , t)$
{\emph{outside}} $B_r (p)$ by the energy minimizing map with the
same boundary values gives a homotopic sweepout $\tilde{\sigma}$.
Moreover, the entire image of $\tilde{\sigma} (\cdot , t)$ is
contained in the convex ball $B^t$ by the maximum
principle.{\footnote{This follows  from lemma $4.1.3$ in \cite{Jo}
which requires that $\sigma (\cdot , t)$ is homotopic to a map in
$B^t$ and this follows from the small energy bound and the uniform
lower bound for the energy of any homotopically non-trivial map
from $\SS^2$ given, e.g., in the first line of the proof of
proposition $2$ on page $143$ of \cite{SY}.}}  It follows that
$\tilde{\sigma}$ is homotopically trivial by contracting each
$\tilde{\sigma} (\cdot , t)$ to the point $\sigma (p, t)$ via a
geodesic homotopy.

\subsection{Uniform continuity of energy improvement on $W^{1,2}$}

It will be convenient to introduce some notation for the next
lemma.  Namely, given a $C^0\cap W^{1,2}$ map $u$ from $\SS^2$ to
$M$ and a finite  collection $\cB$ of disjoint closed balls in
$\SS^2$ so the energy of $u$ on $\cup_{\cB} B$ is at most
$\epsilon_1/3$, let $H(u,\cB):\SS^2 \to M$ denote the map that
coincides with $u$ on $\SS^2 \setminus \cup_{\cB} B$ and on
$\cup_{\cB} B$ is equal to the energy minimizing map from
$\cup_{\cB} B$ to $M$ that agrees with $u$ on $ \cup_{\cB}
\partial B$. To keep the notation simple, we will set $H(u, \cB_1
, \cB_2) = H( H(u, \cB_1) , \cB_2)$. Finally, if $\alpha \in
(0,1]$, then $\alpha \cB$ will denote the collection of concentric
balls but whose radii are shrunk by the factor $\alpha$.

In general, $H(u, \cB_1 , \cB_2)$ is not  the same as $H(u, \cB_2
, \cB_1)$.  This matters in the proof of Theorem \ref{p:tilde},
where harmonic replacement on either $\frac{1}{2} \cB_1$ or
$\frac{1}{2} \cB_2$  decreases the energy of $u$ by a definite
amount. The next lemma (see \eqr{e:gap}) shows that the energy
goes down a definite amount regardless of the order that we do the
replacements.  The second inequality bounds the possible decrease
in energy from applying harmonic replacement on $H(u,\cB_1)$ in
terms of the possible decrease from harmonic replacement on $u$.

\begin{Lem}     \label{l:patch}
There is a constant $\kappa > 0$ (depending on $M$) so that if $u:
\SS^2 \to M$ is in $C^0\cap W^{1,2}$ and $\cB_1$, $\cB_2$ are each
finite collections of disjoint closed balls in $\SS^2$ so that the
energy of $u$ on each $\cup_{\cB_i}B$ is at most $\epsilon_1 / 3$,
then
\begin{equation}    \label{e:gap}
    \Energy (u) - \Energy \left[ H (u, \cB_1 , \cB_2) \right]
     \geq \kappa \,
 \left(  \Energy (u) - \Energy \left[   H(u, \frac{1}{2} \, \cB_2)  \right] \right)^2   \, .
\end{equation}
Furthermore,  for any   $\mu \in [1/8,1/2]$, we have
\begin{equation}    \label{e:gap2}
    \frac{\left( \Energy (u) - \Energy \left[ H (u, \cB_1) \right]
    \right)^{1/2}}{\kappa} +
    \Energy (u) - \Energy \left[ H (u, 2 \, \mu \cB_2) \right]
     \geq  \Energy \left[ H(u , \cB_1) \right] - \Energy \left[   H(u, \cB_1 , \mu \, \cB_2)  \right]     \, .
\end{equation}
\end{Lem}

We will prove Lemma \ref{l:patch} by constructing comparison maps
with the same boundary values and using the minimizing property of
small energy harmonic maps to get upper bounds for the energy.
 The following lemma will be used to construct the comparison maps.

\begin{Lem} \label{l:approx}
There exists $\tau > 0$ (depending on $M$) so that if
$f,g:\partial B_R \to M$ are $C^0 \cap W^{1,2}$ maps that agree at
one point and satisfy
\begin{equation}    \label{e:needthis}
    R \, \int_{\partial B_R}  | f' - g'|^2  \leq \tau^2 \, ,
\end{equation}
then there exists some $\rho \in (0,R/2]$ and a $C^{0} \cap
W^{1,2}$ map $w: B_R \setminus B_{R-\rho} \to M$ so that
\begin{equation}    \label{e:equals}
    w (R-\rho, \theta ) = f (R,\theta)
    {\text{ and }}
    w (R,\theta) = g (R,\theta)  \, ,
\end{equation}
and $ \int_{B_R \setminus
    B_{R-\rho}} |\nabla  {w}|^2 \leq 17 \sqrt{2} \, \left( R \, \int_{\partial B_R}( | f'|^2 +
    |g'|^2)
\right)^{1/2} \,
 \left( R\, \int_{\partial B_R} | f' -
g'|^2 \right)^{1/2}$.
\end{Lem}

\begin{proof}
Let $\Pi$ and $\delta > \hat{\delta} > 0$ (depending on $M$) be as
in the proof of Corollary \ref{c:trivmap2} and
  set $\tau =
\hat{\delta}/\sqrt{2\pi}$.  Since $f-g$ vanishes somewhere on
$\partial B_R$, integrating \eqr{e:needthis} gives $\max |f-g|
\leq \hat{\delta}$.

Since the statement is scale-invariant, it suffices to prove the
case $R=1$.  Set $\rho^2 = \int_{\SS^1}
    |f'-g'|^2  / [8\, \int_{\SS^1} (|f'|^2 + |g'|^2)] \leq 1/4$ and define  $\hat{w}:
B_1 \setminus B_{1-\rho} \to \RR^N $  by
\begin{equation}
    \hat{w} (r,\theta) =
f(\theta) + \left( \frac{r + \rho -1}{\rho} \right) \, \left(
g(\theta) - f(\theta) \right) \, .
\end{equation}
Observe that $\hat{w}$ satisfies \eqr{e:equals}.  Furthermore,
 since $f-g$ vanishes somewhere on $\SS^1$, we
can use Wirtinger's inequality $\int_{\SS^1} |f-g|^2 \leq 4\,
\int_{\SS^1} |(f-g)'|^2$ to bound  $\int_{B_1 \setminus
B_{1-\rho}} |\nabla \hat{w}|^2$ by
\begin{align}
    \int_{B_1 \setminus
    B_{1-\rho}} |\nabla \hat{w}|^2 &\leq \int_{1-\rho}^1 \left[
\frac{1}{\rho^{2}} \, \int_{0}^{2\pi}
    |f-g|^2(\theta) \, d\theta + \frac{1}{r^2} \, \int_{0}^{2\pi} (|f'|^2 + |g'|^2)(\theta) \, d\theta \right] \,
    r\, dr  \notag \\
    &\leq \frac{4}{\rho} \, \int_{0}^{2\pi}
    |f'-g'|^2(\theta) \, d\theta  + 2 \rho \, \int_{0}^{2\pi} (|f'|^2 + |g'|^2)(\theta) \, d\theta
    \\ &= 17/\sqrt{2} \, \left( \int_{\SS^1}
    |f'-g'|^2 \, \int_{\SS^1} (|f'|^2 +
    |g'|^2) \right)^{1/2}
      \, . \notag
\end{align}
 Since $|f-g| \leq
\hat{\delta}$, the image of $\hat{w}$ is contained in
$\overline{M_{\hat{\delta}}}$ where we have  $|d\Pi|^2 \leq 2$.
Therefore, if we  set $w = \Pi \circ \hat{w}$, then the energy of
$w$ is at most twice the energy of $\hat{w}$.
\end{proof}

\begin{proof}
(of Lemma \ref{l:patch}.) We will index the balls in $\cB_1$ by
$\alpha$ and use $j$ for the balls in $\cB_2$; i.e.,  let $\cB_1 =
\{ B^1_{\alpha} \}$ and $\cB_2 = \{ B^2_{j} \}$.  The key point is
that, by Corollary \ref{c:trivmap2}, small
 energy harmonic maps minimize energy.  Using this, we get upper bounds for the
 energy of the harmonic replacement by cutting and pasting to construct
 comparison functions with the same boundary values.

 Observe that the total energy of $u$ on the union of the balls in
 $\cB_1 \cup \cB_2$ is at most $2\epsilon_1/3$.  Since harmonic
 replacement on $\cB_1$
 does not change the map outside these balls and
 is energy non-increasing, it follows that the total energy of
 $H(u, \cB_1)$ on $\cB_2$ is at most $2\epsilon_1/3$.

\noindent {\bf{The proof of \eqr{e:gap}}}.   We will divide
$\cB_2$ into two disjoint subsets, $\cB_{2,+}$ and  $\cB_{2,-}$,
and argue separately, depending on which of these accounts for
more of the decrease in energy after harmonic replacement. Namely,
set
\begin{equation}
    \cB_{2,+} = \{
B^2_{j} \in \cB_2 \, | \, \frac{1}{2} \, B^2_{j} \subset
B^1_{\alpha} {\text{ for some }} B^1_{\alpha} \in \cB_1 \}
{\text{ and }} \cB_{2,-} = \cB_2 \setminus \cB_{2,+} \, .
\end{equation}
Since the balls in $\cB_2$ are disjoint, it follows that
\begin{equation}
    \Energy (u) - \Energy (  H(u, \frac{1}{2} \, \cB_2)
    ) = \left( \Energy (u) - \Energy (  H(u, \frac{1}{2} \, \cB_{2, -})
    ) \right) +    \left(
    \Energy (u) - \Energy (   H(u, \frac{1}{2} \, \cB_{2,+})
   ) \right)  \, .
\end{equation}

\noindent {\bf{Case 1}}. Suppose  that $
    \Energy (u) - \Energy \left[   H(u, \frac{1}{2} \, \cB_{2,+})
    \right]   \geq  \left(
    \Energy (u) - \Energy \left[   H(u, \frac{1}{2} \, \cB_{2})
    \right] \right)/2$.  Since the balls in $\frac{1}{2} \cB_{2,+}$ are contained in
balls in $\cB_1$ and  harmonic replacements minimize energy, we
get
\begin{equation}    \label{e:gap+2}
      \Energy (   H (u, \cB_1 , \cB_2)
    ) \leq  \Energy (  H (u, \cB_1)
  )
     \leq     \Energy (   H(u, \frac{1}{2} \, \cB_{2,+})  )   \, ,
\end{equation}
so that $\left(
    \Energy (u) - \Energy \left[   H(u, \frac{1}{2} \, \cB_{2})
    \right] \right)/2 \leq \Energy (u) -  \Energy (   H(u, \frac{1}{2} \, \cB_{2,+})
) \leq  \Energy (u) -  \Energy (   H (u, \cB_1 , \cB_2))$.

\noindent {\bf{Case 2}}.  Suppose now that
\begin{equation}    \label{e:casetwo}
    \Energy (u) - \Energy (  H(u, \frac{1}{2} \, \cB_{2,-})
   )  \geq  \frac{1}{2} \, \left(
    \Energy (u) - \Energy (   H(u, \frac{1}{2} \, \cB_{2})
    ) \right)  \, .
\end{equation}
Let $\tau
> 0$ be given by Lemma \ref{l:approx}.  We can assume that
\begin{equation}    \label{e:tau}
   9 \, \int_{\SS^2} |\nabla H(u,\cB_1) - \nabla u|^2 \leq \tau^2  \, ,
\end{equation}
since otherwise Theorem \ref{l:trivmap} gives \eqr{e:gap} with
$\kappa = \tau^2/\epsilon_1^2$.
      The key   is to
show for $B^2_j \in \cB_{2,-}$ that
 \begin{align}    \label{e:gapa}
    &\int_{B^2_j} |\nabla H (u, \cB_1)|^2  -  \int_{B^2_j} \left| \nabla H (u, \cB_1 , B^2_j)
    \right|^2
     \geq   \int_{\frac{1}{2} B^2_j} |\nabla u|^2  - \int_{\frac{1}{2} B^2_j}
      \left| \nabla   H(u, \frac{1}{2} \, B^2_j)  \right|^2
        \\ &\quad \quad - C \, \left( \int_{B^2_j} |\nabla u|^2 + \left| \nabla
      H(u, \cB_1)\right|^2
      \right)^{1/2} \,
 \left(   \int_{B^2_j} \left|\nabla (u -
      H(u, \cB_1)) \right|^2 \right)^{1/2}   \,
 , \notag
\end{align}
where $C$ is a universal constant.  Namely,  summing \eqr{e:gapa}
over $\cB_{2,-}$ and using the inequality $\left| \sum a_j \, b_j
\right| \leq \left( \sum a_j^2 \right)^{1/2} \, \left( \sum b_j^2
\right)^{1/2}$,  the bound for the energy of $u$ in  $\cB_1 \cup
\cB_2$, and Theorem \ref{l:trivmap} to relate the energy of $u -
      H(u, \cB_1)$ to $\Energy (u) - \Energy (
      H(u, \cB_1))$
gives
\begin{align}
    \Energy ( u)   - \Energy (H(u, \frac{1}{2} \, \cB_{2,-})
     ) &\leq   \Energy ( H (u, \cB_1))   -  \Energy ( H (u, \cB_1 ,
    \cB_{2,-}))
     + C \, \epsilon_1^{1/2} \,
 \left(   \Energy (u )-
      \Energy [H(u, \cB_1)]  \right)^{1/ 2}    \notag  \\
      &\leq  \delta_{\Energy}
     + C \, \epsilon_1^{1/2} \,
 \delta_{\Energy}^{1/ 2} \leq   (C+1) \, \epsilon_1^{1/2} \,
 \delta_{\Energy}^{1/ 2} \, , \label{e:gapa+}
\end{align}
where we have set $\delta_{\Energy} =  \Energy ( u)   - \Energy (
H (u, \cB_1 ,
    \cB_2))$ in the last line and the last inequality used that
    $\delta_{\Energy} \leq 2\epsilon_1/3 < \epsilon_1$.
  Combining \eqr{e:casetwo} with
    \eqr{e:gapa+}   gives
    \eqr{e:gap}.

To complete Case 2, we must prove \eqr{e:gapa}.  After
translation, we can assume that $B_j^2$ is the ball $B_{R}$ of
radius $R$ about $0$ in $\RR^2$. Set $u_1 = H(u,\cB_1)$ and apply
the co-area formula to get $r \in [3R/4 ,R]$ (in fact, a set of
$r$'s of measure at least $R/36$)  with
\begin{align}    \label{e:chr1}
 \int_{\partial B_r} |\nabla u_1 - \nabla u|^2 &\leq \frac{9}{R}
 \,
 \int_{3R/4}^R \left( \int_{\partial B_s} |\nabla u_1 - \nabla u|^2
 \right) \, ds \leq
 \frac{9}{r}  \, \int_{B_R}
|\nabla u_1 - \nabla u|^2   \, , \\
\label{e:chr2}
 \int_{\partial B_r}  ( |\nabla u_1|^2 + |\nabla u|^2 )  &\leq
 \frac{9}{R}   \int_{3R/4}^R \left(  \int_{\partial B_s}   |\nabla u_1|^2 + |\nabla
 u|^2
\right) \, ds \leq
 \frac{9}{r}  \int_{B_R}
  (|\nabla u_1|^2 + |\nabla u|^2 )  \, .
\end{align}
Since $B_j^2 \in \cB_{2,-}$ and $r >R/2$, the circle $\partial
B_r$ is not contained in any of the balls in $\cB_1$.  It follows
that $\partial B_r$ contains at least one point outside
$\cup_{\cB_1}B$ and, thus,
 there is a point
in $\partial B_r$ where $u = u_1$.
  This and \eqr{e:tau} allow us to apply
  Lemma \ref{l:approx} to get $\rho
\in (0,r/2]$ and a map $w : B_{r} \setminus B_{r-\rho} \to M$ with
$w(r,\theta) = u_1(r,\theta)$,  $w(r-\rho,\theta) = u(r,\theta)$,
 and
\begin{equation}  \label{e:compa}
    \int_{B_r\setminus B_{r-\rho}} |\nabla w|^2 \leq C \, \left( \int_{B^2_j} |\nabla u|^2 + \left| \nabla
      H(u, \cB_1)\right|^2
      \right)^{1/2} \,
 \left(   \int_{B^2_j} \left|\nabla (u -
      H(u, \cB_1)) \right|^2 \right)^{1/2}
      \, .
 \end{equation}
Observe that the map $x \to H(u,B_r)( r\,x/(r-\rho))$ maps
$B_{r-\rho}$ to $M$ and agrees with $w$ on $\partial B_{r-\rho}$.
Therefore, the map from $B_R$ to $M$ which is equal to $u_1$ on
$B_{R} \setminus B_{r}$, is equal to $w$ on $B_{r} \setminus
B_{r-\rho}$, and is equal to $H(u,B_r)( r\, \cdot /(r-\rho))$ on
$B_{r-\rho}$ gives an upper bound for the energy of $H(u_1,B_R)$
\begin{equation}    \label{e:onemore}
    \int_{B_R}  |\nabla
     H(u_1, B_R)|^2 \leq \int_{B_R\setminus B_r} |\nabla u_1|^2 +
     \int_{B_r\setminus B_{r-\rho}} |\nabla w|^2 +
      \int_{B_r}  |\nabla H(u,B_r)|^2
      \, .
\end{equation}
Using \eqr{e:compa} and that $  \left| |\nabla u_1|^2-|\nabla
u|^2\right| \leq
  (|\nabla u| + |\nabla  u_1|)\,  |\nabla (u-u_1)|$, we get
\begin{align}
     &\int_{B_R} |\nabla u_1|^2 -  \int_{B_R}   |\nabla
     H(u_1 , B_R)|^2 \geq \int_{B_r} |\nabla u_1|^2 -  \int_{B_r} |\nabla H(u,B_r)|^2 - \int_{B_r\setminus
     B_{r-\rho}}
      |\nabla
     w|^2  \notag \\
     &\quad   \geq \int_{B_r} |\nabla u|^2 -  \int_{B_r}  |\nabla
     H(u,B_r)|^2
      - C \, \left( \int_{B_r} |\nabla u|^2 + \left| \nabla
      u_1 \right|^2
      \right)^{1/2} \,
 \left(   \int_{B_r} \left|\nabla (u -
      u_1) \right|^2 \right)^{1/2}
    \, . \notag
\end{align}
Since  $ \int_{B_{R/2}} |\nabla
     H(u,B_{R/2})|^2 \leq
\int_{B_{R/2} \setminus B_r} |\nabla u|^2 + \int_{B_r} |\nabla
     H(u,B_r)|^2$, we get
 \eqr{e:gapa}.

\noindent {\bf{The proof of \eqr{e:gap2}}}. We will argue
similarly with a few small modifications that we will describe.
This time, let $\cB_{2,+} \subset \cB_2$ be the balls $B^2_{j}$
with $\mu B^2_{j}$ contained in some $B^1_{\alpha} \in \cB_1 $. It
follows that harmonic replacement on $\mu \cB_{2,+}$ does not
change $H(u,\cB_1)$ and, thus, \begin{equation} \label{e:realeasy}
    \Energy \left[ H(u , \cB_1) \right] =
\Energy \left[ H(u, \cB_1 , \mu \cB_{2,+})  \right] \, .
\end{equation}
Again, we can assume that \eqr{e:tau}  holds. Suppose now that
$B^2_j \in \cB_{2,-}$. Arguing
 as in the proof of \eqr{e:gapa} (switching the roles of $u$ and $H(u, \cB_1)$), we get
\begin{align}    \label{e:gapac2}
    &\int_{B^2_j} |\nabla u|^2  -  \int_{B^2_j} \left| \nabla H (u,  2 \mu B^2_j)
    \right|^2
     \geq   \int_{\mu B^2_j} |\nabla H(u,\cB_1)|^2  - \int_{\mu B^2_j}
      \left| \nabla   H(u, \cB_1 , \mu \, \cB^2_j)  \right|^2
        \\ &\quad \quad - C \, \left( \int_{B^2_j} |\nabla u|^2 + \left| \nabla
      H(u, \cB_1)\right|^2
      \right)^{1/2} \,
 \left(   \int_{B^2_j} \left|\nabla (u -
      H(u, \cB_1)) \right|^2 \right)^{1/2}   \,
 . \notag
\end{align}
Summing this over $\cB_{2,-}$ and arguing as for \eqr{e:gapa+}
gives
\begin{align}    \label{e:gapa+c}
    &\int |\nabla u|^2  -  \int \left| \nabla H (u, 2 \mu \cB_2)
    \right|^2
     \geq  \int  |\nabla H(u,\cB_1)|^2  - \int
      \left| \nabla   H(u, \cB_1 , \mu \, \cB_{2,-})  \right|^2
       \\ &\quad \quad - C \, \epsilon_1^{1/2} \,
 \left(   \Energy (u )-
      \Energy [H(u, \cB_1)]  \right)^{1/2}   \,
 . \notag
\end{align}
Combining \eqr{e:realeasy} and \eqr{e:gapa+c} completes the proof.
\end{proof}

\subsection{Constructing the map from $\tilde{\gamma}$ to $\gamma$}

We will construct ${\gamma}(\cdot , t)$ from $\tilde{\gamma}(\cdot
, t)$ by  harmonic replacement on a  family of balls in $\SS^2$
  varying continuously in $t$. The balls will be
chosen in Lemma \ref{l:goodballs} below. Throughout this
subsection, $\epsilon_1
> 0$ will be the small energy constant (depending on $M$) given by
Theorem \ref{l:trivmap}.

Given $\sigma \in \Omega$ and $\epsilon \in (0, \epsilon_1]$,
define
  the maximal improvement  from harmonic replacement
on families of balls with  energy at most $\epsilon$ by
\begin{equation}
    e_{\sigma,\epsilon}(t) \, = \, \sup_{\cB} \, \{
    \Energy (\sigma (\cdot , t)) - \Energy (H(\sigma (\cdot , t), \frac{1}{2} \cB))
    \} \, ,
\end{equation}
where the supremum is over all finite collections $\cB$ of
disjoint closed balls
    where the total  energy of $\sigma (\cdot , t)$ on $\cB$ is at most
    $\epsilon$. Observe that $e_{\sigma,\epsilon}(t)$ is nonnegative, monotone non-decreasing in $\epsilon$,
    and is positive  if $\sigma (\cdot ,
    t)$ is not harmonic.

\begin{Lem}     \label{l:sf}
If $\sigma (\cdot , t)$ is not harmonic and $\epsilon \in
(0,\epsilon_1]$, then
 there is an open interval $I^t$ containing $t$ so that
 $e_{\sigma,\epsilon/2}(s)  \leq 2 \, e_{\sigma,\epsilon}(t)$ for all
 $s$ in the double interval
 $2 I^t$.
\end{Lem}

\begin{proof}
By \eqr{e:trivmap2} in Corollary \ref{c:trivmap2}, there exists
$\delta_1 > 0$ (depending on $t$) so that if
\begin{equation}    \label{e:it1}
    || \sigma (\cdot , t) - \sigma ( \cdot , s) ||_{C^0 \cap W^{1,2}} <
    \delta_1
\end{equation}
and $\cB$ is a finite collection of disjoint closed balls where
both $\sigma (\cdot , t)$ and $\sigma ( \cdot , s)$ have energy at
most $\epsilon_1$, then
\begin{equation}    \label{e:mu5}
    \left|  \Energy (H(\sigma (\cdot , s), \frac{1}{2}
    \cB))  - \Energy (H(\sigma (\cdot , {{t}}), \frac{1}{2}
    \cB)) \right| \leq
    e_{\sigma,\epsilon}(t) /2  \, .
\end{equation}
Here we have used that  $e_{\sigma,\epsilon}(t)  > 0$ since
$\sigma (\cdot , t)$ is not harmonic.  Since $t \to \sigma (\cdot
, t)$ is continuous as a map to $C^0 \cap W^{1,2}$, we can choose
$I^t$ so that for all $s
 \in 2 \,I^t$ \eqr{e:it1} holds  and
 \begin{equation}   \label{e:mu6}
    \frac{1}{2} \, \int_{\SS^2} \left| |\nabla \sigma (\cdot , t)|^2  - |\nabla \sigma ( \cdot , s)|^2 \right| \leq
      \min \,  \{
     \frac{\epsilon}{2} , \frac{e_{\sigma,\epsilon}(t)}{2}     \} \, .
\end{equation}
Suppose now that $s \in 2 I^t$ and the energy of $\sigma ( \cdot ,
s)$ is at most $\epsilon/2$ on a collection $\cB$.  It follows
from \eqr{e:mu6}  that the energy of $\sigma ( \cdot , t)$ is at
most $\epsilon$ on $\cB$.  Combining \eqr{e:mu5} and \eqr{e:mu6}
gives
\begin{equation}    \label{e:mu7}
    \left|  \Energy (\sigma (\cdot , s)) - \Energy (H(\sigma (\cdot , s), \frac{1}{2}
    \cB)) - \Energy (\sigma (\cdot , t)) + \Energy (H(\sigma (\cdot , {{t}}), \frac{1}{2}
    \cB)) \right| \leq
    e_{\sigma,\epsilon}(t)   \, .
\end{equation}
Since this applies to any such $\cB$, we get that
$e_{\sigma,\epsilon/2}(s)  \leq 2 \, e_{\sigma,\epsilon}(t)$.
\end{proof}

Given a sweepout with no harmonic slices, the next lemma
constructs finitely many  collections of balls so that harmonic
replacement on at least one of these collections strictly
decreases the energy.  In addition,
 each collection consists of finitely many pairwise disjoint closed
balls.

\begin{Lem}     \label{l:goodballs}
If $W>0$ and  $\tilde{\gamma} \in \Omega$ has no non-constant
harmonic slices, then we get an integer $m$ (depending on
$\tilde{\gamma}$),   $m$ collections of balls $\cB_1, \dots ,
\cB_m$ in $\SS^2$, and continuous functions $r_1 , \dots , r_{m}:
[0,1] \to [0,1]$ so that for each $t$:
\begin{enumerate}
\item[(1)] At most two $r_j(t)$'s are positive and
   $   \sum_{B \in \cB_j} \, \frac{1}{2} \,\int_{ r_j(t) B} \, |\nabla \tilde{\gamma} (\cdot , t) |^2 <
    \epsilon_1 /3 $ for each $j$.
 \item[(2)] If
 $\Energy (\tilde{\gamma} (\cdot , t)) \geq W/2$, then
 there exists $j(t)$ so that harmonic replacement on
  $   \frac{r_{j(t)}}{2} \, \cB_{j(t)} $
  decreases energy by at least
$e_{\tilde{\gamma}, \epsilon_1/8}(t)/8 $.
\end{enumerate}
\end{Lem}

\begin{proof}
Since the energy of the slices is continuous in $t$, the set $I =
\{ t   \, | \, \Energy (\tilde{\gamma} (\cdot , t)) \geq W/2 \}$
is compact.
  For each $t \in I$, choose a finite  collection $\cB^t$ of disjoint closed balls  in $\SS^2$
  with
$\frac{1}{2} \, \int_{\cup_{\cB^t}} |\nabla   \tilde{\gamma}
(\cdot , t)|^2 \leq \epsilon_1/4 $ so
\begin{equation}    \label{e:godown}
    \Energy (\gamma (\cdot , t)) - \Energy (H (\gamma (\cdot , t)
    , \frac{1}{2} \, \cB^t)) \geq \frac{ e_{\tilde{\gamma},\epsilon_1/4}(t)}{2} >
    0\, .
\end{equation}
Lemma
 \ref{l:sf} gives an open
interval $I^t$ containing $t$ so that for all
 $s \in
 2 I^t$
 \begin{equation} \label{e:godowna}
    e_{\tilde{\gamma},\epsilon_1 /8}(s) \leq 2 \,
    e_{\tilde{\gamma},\epsilon_1/4}(t)\, .
 \end{equation}
   Using the continuity of $\tilde{\gamma}(\cdot , s)$ in $C^0
\cap W^{1,2}$ and  Corollary \ref{c:trivmap2}, we can shrink $I^t$
so that    $\tilde{\gamma}(\cdot , s)$ has energy at most
$\epsilon_1/3$ in $\cB^t$ for  $s \in 2I^t$ and, in addition,
\begin{equation}    \label{e:godown2}
    \left|  \Energy (\gamma (\cdot , s)) - \Energy (H(\gamma (\cdot , s), \frac{1}{2}
    \cB^t)) - \Energy (\gamma (\cdot , t)) + \Energy (H(\gamma (\cdot , {{t}}), \frac{1}{2}
    \cB^t)) \right| \leq
   \frac{ e_{\tilde{\gamma},\epsilon_1/4}(t)}{4 }  \, .
\end{equation}

Since $I$ is compact, we can cover $I$ by finitely many $I^t$'s,
say  $I^{t_1} , \dots , I^{t_m}$.  Moreover, after discarding some
of the intervals, we can arrange that each $t$ is in at least one
closed interval $\overline{I^{t_j}}$, each $\overline{I^{t_j}}$
intersects at most two other $\overline{I^{t_k}}$'s, and the
$\overline{I^{t_k}}$'s intersecting $\overline{I^{t_j}}$ do not
intersect each other.{\footnote{We will give a recipe for doing
this. First, if $\overline{I^{t_1}}$ is contained in the union of
two other intervals, then throw it out. Otherwise, consider the
intervals whose left endpoint is in $\overline{I^{t_1}}$, find one
whose right endpoint is largest and discard the others (which are
anyway contained in these). Similarly, consider the intervals
whose right endpoint is in $\overline{I^{t_1}}$ and throw out all
but one whose left endpoint is smallest.  Next, repeat this
process on $I^{t_2}$ (unless it has already been discarded), etc.
 After at most $m$ steps, we get the desired cover. \label{fn:dim}}}  For each $j=1, \dots m$, choose a continuous function
$r_j:[0,1] \to [0,1]$ so that
\begin{itemize}
\item $r_j (t) =   1$ on $\overline{I^{t_j}}$ and $r_j
(t)$ is zero for $t \notin 2I^{t_j}$. \item $r_j(t)$ is zero on
the intervals that {\emph{do not}} intersect $\overline{I^{t_j}}$.
\end{itemize}
Property (1) follows directly and (2) follows  from
\eqr{e:godown}, \eqr{e:godowna}, and \eqr{e:godown2}.
\end{proof}

\begin{proof}
(of Theorem \ref{p:tilde}).  Let
  $\cB_1 , \dots , \cB_{m}$
and   $r_1 , \dots , r_{m}: [0,1] \to [0,\pi)$ be given by Lemma
\ref{l:goodballs}.  We will use an $m$ step replacement process to
define ${\gamma}$.  Namely, first set $\gamma^0 = \tilde{\gamma}$
and then, for each $k=1, \dots , m$, define $\gamma^k$ by applying
harmonic replacement to $\gamma^{k-1} (\cdot , t)$ on the $k$-th
family of balls $r_k(t) \, \cB_{k}$; i.e, set $\gamma^k (\cdot ,
t) = H( \gamma^{k-1} (\cdot , t), r_k(t) \, \cB_k)$. Finally, we
set ${\gamma} = \gamma^m$.

A key point in the construction is that property (1) of the family
of balls gives that only two $r_k(t)$'s are positive for each $t$.
Therefore, the energy bound on the balls given by property (1)
implies that each energy minimizing map replaces a map with energy
at most $2\epsilon_1/3 < \epsilon_1$.   Hence, Corollary
\ref{c:trivmap2} implies that these depend continuously on the
boundary values, which are themselves continuous in $t$, so that
the resulting map $\tilde{\gamma}$ is also continuous in $t$.
Finally, it is clear that $\tilde{\gamma}$ is homotopic to
$\gamma$ since continuously shrinking the disjoint closed balls on
which we make harmonic replacement gives an explicit homotopy.
Thus, $\gamma \in \Omega_{\tilde{\gamma}}$ as claimed.

 For each $t$ with $\Energy (\tilde{\gamma}
(\cdot , t)) \geq W/2$, property (2) of the family of balls gives
some $j(t)$ so that harmonic replacement for $\tilde{\gamma}
(\cdot , t)$ on $\frac{r_j(t)}{2} \, \cB_{j(t)}$ decreases the
energy by at least $\frac{e_{\tilde{\gamma},
\epsilon_1/8}(t)}{8}$. Thus, even in the worst case where $ r_j(t)
\, \cB_{j(t)}$ is the second family of balls that we do
replacement on at $t$, \eqr{e:gap} in Lemma \ref{l:patch} gives
\begin{equation}    \label{e:lastneed}
    \Energy (\tilde{\gamma} (\cdot , t)) -
    \Energy ( {\gamma} (\cdot , t)) \geq  \kappa \, \left( \frac{e_{\tilde{\gamma}, \epsilon_1/8}(t)}{8} \right)^{2} \,
    .
\end{equation}
 To establish ($B_{\Psi}$), suppose that $\cB$ is a finite
collection of disjoint closed balls in $\SS^2$ so that the energy
of $\gamma (\cdot , t)$ on $ \cB $ is at most $\epsilon_1/12$.  We
can assume that   $\gamma^k (\cdot , t)$ has energy at most
$\epsilon_1/8$ on $\cB$ for every $k$ since otherwise Theorem
\ref{l:trivmap} implies a positive lower bound for $\Energy
(\tilde{\gamma} (\cdot , t)) -
    \Energy ( {\gamma} (\cdot , t))$.  Consequently, we can
apply  \eqr{e:gap2} in Lemma \ref{l:patch} twice (first with $\mu
= 1/8$ and then with $\mu =1/4$) to get
\begin{align}    \label{e:gap3}
    \Energy ({\gamma} (\cdot , t)) - \Energy \left[ H ({\gamma} (\cdot , t),  \frac{1}{8} \,
    \cB) \right] &\leq \Energy (\tilde{\gamma} (\cdot , t)) - \Energy \left[ H (\tilde{\gamma} (\cdot , t),  \frac{1}{2} \,
    \cB) \right] +   \frac{2}{\kappa} \,
    \left( \Energy (\tilde{\gamma} (\cdot , t)) -
    \Energy ( {\gamma} (\cdot , t)) \right)^{1/2}
     \notag \\
& \leq
 e_{\tilde{\gamma}, \epsilon_1/8}(t) +  \frac{2}{\kappa} \,
   \left( \Energy (\tilde{\gamma} (\cdot , t)) -
    \Energy ( {\gamma} (\cdot , t)) \right)^{1/2}  \, .
\end{align}
Combining \eqr{e:lastneed} and \eqr{e:gap3} with Theorem
\ref{l:trivmap} gives ($B_{\Psi}$) and, thus, completes the proof.
\end{proof}

\appendix

\section{Bubble convergence implies varifold convergence}
\label{a:A}

\subsection{Bubble convergence and the topology on $\Omega$}

We will need a notion of convergence for  a sequence $v^j$ of
$W^{1,2}$ maps    to  a collection $\{ u_0 , \dots , u_m \}$ of
$W^{1,2}$ maps which is similar in spirit to the convergence in
Gromov's compactness theorem for pseudo holomorphic curves,
\cite{G}.  The notion that we will use is a slight weakening of
the bubble tree convergence developed by Parker and Wolfson for
$J$-holomorphic curves in \cite{PaW} and used by Parker for
harmonic maps in \cite{Pa}.  In our applications, the $v^j$'s will
be approximately harmonic while the limit maps $u_i$ will be
harmonic.   We will need the next definition to make this precise.

$S^+$ and $S^-$ will denote the northern and southern hemispheres
in $\SS^2$ and $p^+ = (0,0,1)$ and $p^- = (0,0,-1)$ the north and
south poles.

\begin{Def}
 Given a ball $B_r(x) \subset \SS^2$,  the {\emph{conformal
dilation}} taking $B_r(x)$ to $S^-$ is the composition of
translation $x \to p^-$ followed by dilation of $\SS^2$ about
$p^-$  taking $B_r(p^-)$ to $S^-$.
\end{Def}

The standard example of a conformal dilation  comes from applying
stereographic projection $\Pi :\SS^2 \setminus \{ (0,0,1) \} \to
\RR^2$,  then dilating $\RR^2$ by a positive  $\lambda \ne 1$, and
  applying $\Pi^{-1}$.

In the definition below of convergence, the map $u_0$ will be the
standard $W^{1,2}$-weak limit of the $v^j$'s (see (B1)), while the
other $u_i$'s will arise as weak limits of the composition of the
$v^j$'s with a divergent sequence of conformal dilations of
$\SS^2$ (see (B2)).
  The condition (B3)
guarantees that these limits all arise in genuinely distinct ways,
and the condition (B4) means that together the $u_i$'s account for
all of the energy.

\begin{Def}     \label{d:bubble}
{\bf{Bubble convergence}}.  We will say that a sequence $v^j:\SS^2
\to M$ of $W^{1,2}$ maps  converges to a collection of $W^{1,2}$
maps  $u_0 , \dots , u_m : \SS^2 \to M$ if the following hold:
\begin{enumerate}
\item[(B1)] The $v^j$'s converge weakly to $u_0$ in $W^{1,2}$ and
there is a finite set $\cS_0 = \{ x_0^1 , \dots , x_0^{k_0} \}
\subset \SS^2$ so that the $v^j$'s converge strongly to $u_0$ in
$W^{1,2}(K)$ for any compact $K \subset \SS^2 \setminus \cS_0$.
\item[(B2)]  For each $i > 0$, we get a point $x_{\ell_i} \in
\cS_0$ and a sequence of balls $B_{r_{i,j}}(y_{i,j})$ with
$y_{i,j} \to x_{\ell_i}$ and $r_{i,j} \to 0$.  Furthermore, if
$D_{i,j} : \SS^2 \to \SS^2$ is the conformal dilation taking the
southern hemisphere to
 $B_{r_{i,j}}(y_{i,j})$, then the maps
$v^j \circ D_{i,j}$ converge to $u_i$ as in (B1).  Namely,  $v^j
\circ D_{i,j} \to u_i$ weakly in $W^{1,2}(\SS^2)$ and there is a
finite set  $\cS_i$ so that the $v^j \circ D_{i,j}$'s converge
strongly in $W^{1,2}(K)$ for any compact $K \subset \SS^2
\setminus \cS_i$.
\item[(B3)] If $i_1 \ne i_2$, then
$\frac{r_{i_1,j}}{r_{i_2,j}} + \frac{r_{i_2,j}}{r_{i_1,j}} +
\frac{|y_{i_1,j} - y_{i_2,j}|^2 }{ r_{i_1,j} \, r_{i_1,j} } \to
\infty$.
 \item[(B4)] We get the energy equality $\sum_{i=0}^m
\Energy (u_i) = \lim_{j \to \infty} \Energy
        (v^j) \,$.
\end{enumerate}
\end{Def}

\subsection{Two simple examples of bubble
convergence}

The simplest non-trivial example of bubble convergence is when
each map $v^j= u \circ \Psi_j$ is the composition of a fixed
harmonic map $u:\SS^2 \to M$ with a divergent sequence of
dilations $\Psi_j :\SS^2 \to \SS^2$.  In this case, the $v^j$'s
converge to the constant map $u_0 = u(p_+)$ on each compact set of
$\SS^2 \setminus \{ p_- \}$  and  all of the energy concentrates
at the single point $p_- = \cS_0$.  Composing the $v^j$'s with the
divergent sequence $\Psi_j^{-1}$ of conformal dilations gives the
limit $u_1 = u$.

 For the second example, let $\Pi :\SS^2 \setminus \{ (0,0,1) \} \to \RR^2$
be stereographic projection and let $z= x+iy$ be complex
coordinates on $\RR^2 = \CC$.  If we set $f_j(z) = 1/(jz) + z =
\frac{z^2 +1/j}{z}$, then the maps $v^j = \Pi^{-1} \circ f_j \circ
\Pi:\SS^2 \to \SS^2$ are conformal and, therefore, also harmonic.
Since each $v^j$ is a rational map of degree two, we have $\Energy
(v^j) =   \Area (v^j) = 8\pi$. Moreover, the $v^j$'s converge away
from $0$ to the identity map which has energy $4\pi$.  The other
$4\pi$ of energy disappears at $0$ but can  be accounted for by a
map $u_1$ by composing with a divergent sequence of conformal
dilations; $u_1$ must also have degree one.  In this case,   the
conformal dilations take $f_j$ to $\tilde{f}_j (z) = f_j (z/j) =
1/z + z/j$ which converges to the conformal inversion about the
circle of radius one.

\subsection{Bubble convergence implies varifold convergence}

\begin{Pro}     \label{l:bubvar}
If a sequence $v^j$  of $W^{1,2}(\SS^2 , M)$ maps bubble converges
to a finite collection of smooth maps $u_0 , \dots , u_m: \SS^2
\to M$, then it also varifold converges.
\end{Pro}

Before getting to the proof, recall that a sequence of functions
$f_j$ is said to {\emph{converge in measure}} to a function $f$ if
for all $\delta > 0$ the measure of $\{ x \, | \, |f_j - f|(x) >
\delta \}$ goes to zero as $j \to \infty$; see \cite{R}, page
$95$. Clearly, $L^1$ convergence implies convergence in measure.
Furthermore, if $f_j \to f$ in measure and $h$ is uniformly
continuous, then $h\circ f_j \to h\circ f $ in measure.  Finally,
we will use the following
 general version of the
dominated convergence theorem which combines  theorem $17$ on page
92 of \cite{R}  and proposition $20$ on page $96$ of \cite{R}:
\begin{enumerate}
\item[(DCT)]
If $f_j \to f$ in measure, $g_j \to g$ in $L^1$, and $|f_j|
\leq g_j$, then $\int f_j \to \int f$.
\end{enumerate}

We will also use that the map $\nabla u \to J_u$ is continuous as
a map from $L^2$ to $L^1$ and, thus, $\Area (u)$ is continuous
with respect to $\Energy (u)$.  To be precise, if $u , v \in
W^{1,2} (\SS^2 , M)$, then
\begin{equation}    \label{e:jinL1}
    \left| J_u - J_v \right| \leq \sqrt{2} \, |\nabla u - \nabla v|^{1/2} \, \max \{  |\nabla
    u|^{3/2} , \, |\nabla v|^{3/2} \} \, .
\end{equation}
This follows from the linear algebra fact{\footnote{Note that
$|S^T T| \leq |S| \, |T|$, $\left| \Tr \, (S^T T) \right| \leq |S|
\, |T|$, and if $X_t$ is a path of $2\times 2$ matrices, then
$\partial_t \, \det X_t = \Tr \, \left( X_t^c \,
\partial_t X_t \right)$ where $X_t^c$ is the cofactor matrix given by swapping
diagonal entries and multiplying off-diagonals by $-1$. Applying
this to $X_t = \left( S + t \, (T - S) \right)^T \, \left( S + t
\, (T - S) \right)$ and using the mean value theorem gives
\eqr{e:linalgf}.}}
 that if $S$
and $T$ are $N\times 2$ matrices, then
\begin{equation}    \label{e:linalgf}
    \left| \det \left( S^T
\, S \right) - \det \left( T^T \, T \right) \right| \leq 2 \,
|T-S| \, \max \{ |S|^{3} , \, |T|^{3} \} \, ,
\end{equation}
where $|S|^2$ is the sum of the squares of the entries of $S$ and
$S^T$ is the transpose.

\begin{proof}
(of Proposition \ref{l:bubvar}.) For each $v^j$, we will let $V^j$
denote the corresponding map to $G_2
 M$.  Similarly, for each $u_i$,
 let $U_i$ denote the corresponding map to $G_2 M$.

 It follows from (B1)-(B4) that we can choose $m+1$ sequences
of domains $\Omega^j_0 , \dots , \Omega^j_m \subset \SS^2$ that
are pairwise
 disjoint for each $j$ and so that for each $i= 0, \dots , m$ applying
$D_{i,j}^{-1}$ to $\Omega^j_i$ gives a sequence of domains
converging to $\SS^2 \setminus \cS_i$ and accounts for all the
energy, that is,
\begin{equation}    \label{e:noloss}
    \lim_{j\to \infty} \int_{\SS^2 \setminus \left( \cup_i \,
    \Omega^j_i \right) } \, \, |\nabla v^j|^2 = 0 \, .
\end{equation}
 By \eqr{e:noloss},    the proposition follows from showing
 for each $i$ and any   $h$   in $C^0(G_2 M)$ that
\begin{equation}    \label{e:needi}
    \int_{\SS^2} h \circ U_i \, J_{u_i} =
\lim_{j \to \infty} \, \int_{ \Omega^j_i } h \circ V^{j}
    \, J_{v^j} =
    \lim_{j \to \infty} \, \int_{D_{i,j}^{-1} \left( \Omega^j_i \right)} h \circ
    V^{j} \circ D_{i,j}
    \, J_{\left( v^j \circ D_{i,j} \right)}   \, ,
\end{equation}
where the last equality is simply the change of variables formula
for integration.

To simplify notation in the proof of \eqr{e:needi},
  for each $i$ and $j$,  let $v_i^j$ denote the restriction of
 $v^j \circ D_{i,j}$ to $D_{i,j}^{-1} \left( \Omega^j_i \right)$ and let $V_i^j$ denote the
 corresponding map to $G_2 M$.

 Observe first that $J_{v_i^j} \to J_{u_i}$ in $L^1(\SS^2)$ by
 \eqr{e:jinL1}. Given $\epsilon > 0$ and $i$, let
 $\Omega_{\epsilon}^i$  be the set where $J_{u_i}
 \geq \epsilon$.  Since $h$ is
 bounded and $J_{v_i^j} \to J_{u_i}$ in $L^1(\SS^2)$,
 \eqr{e:needi} would follow from
 \begin{equation}    \label{e:needi2}
    \lim_{j \to \infty} \, \int_{\Omega_{\epsilon}^i} h \circ V_{i}^{j}
    \, J_{v_i^j} = \int_{\Omega_{\epsilon}^i} h \circ U_i \, J_{u_i} \, .
\end{equation}
However, given any $\delta > 0$, $W^{1,2}$ convergence  implies
that  the measure of
\begin{equation}
    \{ x \in \Omega_{\epsilon}^i \, | \, J_{v_i^j} \geq
    \frac{\epsilon}{2} {\text{ and }} |V_i^j - U_i| \geq \delta \}
\end{equation}
goes to zero as $j \to \infty$.  Since   $L^1$ convergence of
Jacobians implies that the measure of  $\{ x \in
\Omega_{\epsilon}^i \, | \, J_{v_i^j} <
    \frac{\epsilon}{2} \}$ goes to zero, it follows that
 the maps $V_{i}^{j}$ converge in measure to
 $U_i$ on $\Omega_{\epsilon}^i$.   Therefore,  the $h \circ V_{i}^{j}$'s converge in measure to
$h \circ U_i$ on $\Omega_{\epsilon}^i$.
 Consequently, the
 general version of the
dominated convergence theorem (DCT)  gives \eqr{e:needi2} and,
thus, also \eqr{e:needi}.
\end{proof}

\section{The proof of Proposition \ref{p:gl2}}  \label{s:ppp}

The proof of Proposition \ref{p:gl2} will follow the general
structure developed by Parker and Wolfson in \cite{PaW} and used
by Parker in \cite{Pa} to prove compactness of harmonic maps with
bounded energy.  The main difficulty is to rule out loss of energy
in the limit (see (B4) in the definition of bubble convergence).
The rough idea to deal with this is that energy loss only occurs
when there are very small annuli where the maps are ``almost''
harmonic and the ratio between the inner and outer radii of the
annulus is enormous. We will use Proposition \ref{c:c1} to show
that the map must be ``far'' from being conformal on such an
annulus and, thus, condition (A) allows us to rule out energy
loss.  Here ``far'' from conformal will mean that the
$\theta$-energy of the map is much less than the radial energy. To
make this precise, it is convenient to replace an annulus
$B_{\e^{r_2}} \setminus B_{\e^{r_1}}$ in $\RR^2$ by the
conformally equivalent cylinder $[r_1,r_2] \times \SS^1$.  The
(non-compact) cylinder $\RR \times \SS^1$ with the flat product
metric and  coordinates $t$ and $\theta$ will be denoted by $\cC$.
For $r_1 < r_2$, let $\cC_{r_1,r_2} \subset \cC$ be the product
$[r_1,r_2] \times \SS^1$.

\subsection{Harmonic maps on cylinders}

  The main result of this subsection is that harmonic maps with small
energy on long cylinders are almost radial. This implies that a
sequence of such maps with energy bounded away from zero is
uniformly far from being conformal and, thus, cannot satisfy (A)
in Proposition \ref{p:gl2}.  It will be used to prove a similar
result for ``almost harmonic'' maps in Proposition \ref{c:c1} and
eventually be used when we show that energy will not be lost.

\begin{Pro}     \label{c:ann2}
Given $\delta > 0$, there exist $\epsilon_2 > 0$ and  $\ell \geq
1$ depending on $\delta$ (and $M$) so that if $u$ is a
(non-constant) $C^3$ harmonic map from the flat cylinder
$\cC_{-3\ell,3\ell}= [-3\ell,3\ell] \times \SS^1$ to $M$ with
$\Energy (u) \leq \epsilon_{2}$, then
\begin{equation}
    \int_{ \cC_{-\ell,\ell} } |u_{\theta}|^2 < \delta \, \int_{ \cC_{-2\ell,2\ell} } |\nabla
    u|^2
        \, .
\end{equation}
\end{Pro}

To show this proposition, we show a differential inequality which
leads to exponential growth for the $\theta$-energy of the
harmonic map on the level sets of the cylinder. Once we have that,
the proposition follows. Namely, if the $\theta$-energy in the
``middle'' of the cylinder was a definite fraction of the total
energy over the double cylinder, then the exponential growth would
force the $\theta$-energy of   near the boundary of the cylinder
to be too large.

The following standard lemma is the differential inequality for
the $\theta$-energy that leads to exponential growth through Lemma
\ref{l:comp} below.

\begin{Lem}     \label{l:difi}
 For a  $C^3$ harmonic map $u$ from $\cC_{r_1,r_2} \subset \cC$ to
  $M \subset \RR^N$
\begin{equation}
     \partial_t^2 \int_{t} |u_{\theta}|^2 \geq \frac{3}{2} \, \int_{t}
     |u_{\theta}|^2 - 2 \, \sup_M |A|^2 \, \int_{t} |\nabla u|^4 \, .
\end{equation}
\end{Lem}

\begin{proof}
Differentiating $\int_{t} |u_{\theta}|^2$ and integrating by parts
in $\theta$ gives
\begin{align}   \label{e:difi}
    \frac{1}{2} \, \partial_t^2 \int_{t} |u_{\theta}|^2  &= \int_{t} |u_{t
    \theta}|^2 + \int_{t} \langle u_{\theta},  \, u_{tt\theta} \rangle =
    \int_{t} |u_{t
    \theta}|^2 - \int_{t} \langle u_{\theta\theta } , \, u_{tt} \rangle =
\int_{t} |u_{t
    \theta}|^2 - \int_{t} \langle u_{\theta\theta } , \, (\Delta u -
    u_{\theta \theta} ) \rangle
    \notag \\
    &\geq \int_{t} |u_{t
    \theta}|^2 + \int_{t} |u_{\theta\theta }|^2  - \sup_M |A| \, \int_{t} |u_{\theta \theta}| \,
    |\nabla u|^2 \, ,
\end{align}
where the last inequality used that $|\Delta u|\leq |\nabla u|^2
\,  \sup_M |A|$ by the harmonic map equation.{\footnote{If $u^i$
are the components of the harmonic map $u$, $g_{jk}$ is the metric
on $B$, and $A_{u(x)}^i$ is the $i$-th component of the second
fundamental form of $M$ at the point $u(x)$, then page $157$ of
\cite{SY} gives
\begin{equation}    \label{e:footnote}
    \Delta_M u^i = g^{jk} A^i_{u(x)} \left( \partial_j u , \partial_k
    u
    \right)\, .
    \end{equation}
}} The lemma follows from applying the absorbing inequality  $2ab
\leq a^2/2 + 2b^2$ and noting that $\int_{t} u_{\theta} = 0$ so
that Wirtinger's inequality gives $\int_{t} |u_{\theta}|^2 \leq
\int_{t} |u_{\theta  \theta}|^2$.
\end{proof}

\begin{Rem}     \label{r:harmonic}
The differential inequality in Lemma \ref{l:difi} immediately
implies that Proposition \ref{c:ann2} holds for harmonic
functions, i.e., when $|A| \equiv 0$, even without the small
energy assumption.  The general case will follow by using the
small energy assumption to show that the perturbation terms are
negligible.
\end{Rem}

We will need a simple ODE comparison lemma:

\begin{Lem} \label{l:comp}
Suppose that $f$ is a non-negative $C^2$ function on
$[-2\ell,2\ell] \subset \RR$ satisfying
\begin{equation}    \label{e:comp}
f'' \geq   f - a \, ,
\end{equation}
for some constant $a > 0$.  If $\max_{[-\ell,\ell]} \, f \geq 2a$,
then
\begin{equation}    \label{e:ode2}
    \int_{-2\ell}^{2\ell} \, f \geq 2\sqrt{2} \, a \, \sinh (\ell/\sqrt{2})  \, .
\end{equation}
\end{Lem}

\begin{proof}
Fix some $x_0 \in [-\ell,\ell]$ where $f$ achieves its maximum on
$[-\ell,\ell]$.  Since the lemma is invariant under reflection $x
\to -x$, we can assume that $x_0 \geq 0$.  If $x_0$ is an interior
point, then $f'(x_0)=0$; otherwise, if $x_0 = \ell$, then $f'(x_0)
\geq 0$. In either case, we get
    $f'(x_0) \geq 0$.
Since $f(x_0) \geq 2a$, \eqr{e:comp} gives $f''(x_0) \geq a
> 0$ and, hence, $f'$ is strictly increasing at $x_0$.

We claim that $f'(x) > 0$ for all $x$ in $(x_0, 2\ell]$.  If not,
then there would be a first point $y > x_0$ with
$f'(y) = 0$. It follows that $f' \geq 0$ on $[x_0,y]$ so that $f
\geq f(x_0) \geq 2a$ on $[x_0,y]$ and, thus, that $f'' \geq a > 0$
on $[x_0,y]$,
 contradicting that $f'(y) \leq f'(x_0)$.

By the claim,  $f$ is  monotone increasing on $[x_0,2\ell]$ so
that \eqr{e:comp} gives
\begin{equation}    \label{e:fpg1}
f'' \geq \frac{1}{2} \, f  {\text{ on }}  [x_0,2\ell] \, .
\end{equation}
By a standard Riccati comparison argument using $f'(x_0) \geq 0$
and \eqr{e:fpg1} (see, e.g., corollary $A.9$ in \cite{CDM}), we
get for $t \in [0,2\ell-x_0]$
\begin{equation}    \label{e:ode}
f(x_0 + t) \geq f(x_0) \, \cosh (t/\sqrt{2}) \geq  2\,a \, \cosh (
t/\sqrt{2})   \, .
\end{equation}
Finally, integrating \eqr{e:ode} on $[0,\ell]$ gives \eqr{e:ode2}.
\end{proof}

\begin{proof} (of Proposition \ref{c:ann2}.)
Since we will choose $\ell \geq 1$ and $\epsilon_2 <
\epsilon_{SU}$, the small-energy interior estimates for harmonic
maps (see lemma
 $3.4$ in  \cite{SaU}; cf. \cite{SU1}) imply
that
\begin{equation}
    \sup_{\cC_{-2\ell,2\ell}} \, |\nabla u|^2 \leq C_{SU} \,  \int_{\cC_{-3\ell,3\ell}} \, |\nabla
    u|^2 \leq  C_{SU} \, \epsilon_{2} \, .
\end{equation}
Set $f(t) = \int_{t} |u_{\theta}|^2$.  It follows from Lemma
\ref{l:difi}  that
\begin{equation}   \label{e:difim}
    f''(t)  \geq  \frac{3}{2} \, f(t)  - 2\,  \sup_M |A|^2 \, C_{SU} \, \epsilon_2 \, \int_{t}
     (|u_{\theta}|^2 + |u_t|^2)\geq  f(t) - C \, \epsilon_2 \,  \int_{t} (|u_t|^2 - |u_{\theta}|^2)
    \, ,
\end{equation}
where $C = 2\, C_{SU} \, \sup_M |A|^2$ and we have assumed that
$C\, \epsilon_2 \leq 1/4$ in the second inequality.

We will use that  $\int_{t} (|u_t|^2 - |u_{\theta}|^2)$ is
constant in $t$. To see this, differentiate  to get
\begin{equation}    \label{e:difft}
\frac{1}{2} \, \partial_t \, \int_{t} (|u_t|^2 - |u_{\theta}|^2) =
\int_{t} \left( \langle u_t ,u_{tt} \rangle - \langle u_{\theta} ,
u_{t\theta} \rangle \right) = \int_{t} \langle u_t ,( u_{tt} +
u_{\theta \theta})\rangle = 0 \, ,
\end{equation}
where the second equality used integration by parts in $\theta$
and the last equality used that $u_{tt} + u_{\theta \theta} =
\Delta u$ is normal to $M$ while $u_t$ is
tangent.{\footnote{In fact, something much stronger is true:
The complex-valued function
$$
    \phi (t,\theta) = (|u_t|^2 - |u_{\theta}|^2) - 2\,i \, \langle u_t ,
    u_{\theta}\rangle
$$
is holomorphic on the cylinder (see page $6$ of \cite{SY}).  This
is usually called the Hopf differential.}} Bound this
constant by
\begin{equation}    \label{e:216}
    \int_{t} (|u_t|^2 - |u_{\theta}|^2) = \frac{1}{4\ell} \, \int_{\cC_{-2\ell,2\ell}} (|u_t|^2 -
    |u_{\theta}|^2)\leq \frac{1}{4\ell} \, \int_{\cC_{-2\ell,2\ell}}  |\nabla u|^2 \, .
\end{equation}

 By \eqr{e:difim} and \eqr{e:216}, Lemma \ref{l:comp}
 with $a =  \frac{C \, \epsilon_2}{4\ell} \, \int_{\cC_{-2\ell,2\ell}}  |\nabla u|^2$
  implies that either
\begin{equation}    \label{e:p1}
    \max_{[-\ell,\ell]} \, f  < 2 \,   \frac{C \, \epsilon_2}{4\ell} \, \int_{\cC_{-2\ell,2\ell}}  |\nabla u|^2 \, ,
\end{equation}
or
\begin{equation}
    \int_{\cC_{-2\ell,2\ell}}   |  u_{\theta}|^2 = \int_{-2\ell}^{2\ell} \, f(t) \, dt
    \geq  2\sqrt{2} \, C \, \epsilon_2 \,   \frac{\sinh (\ell/\sqrt{2})}{4\ell} \, \int_{\cC_{-2\ell,2\ell}} |\nabla u|^2   \, .
\end{equation}
The second possibility cannot occur as long as $\ell$ is
sufficiently large so that we have
\begin{equation}    \label{e:choosen}
    2\sqrt{2} \, C \, \epsilon_2 \,   \frac{\sinh (\ell/\sqrt{2})}{4\ell} > 1 \, .
\end{equation}
Using the upper bound \eqr{e:p1} for $f$ on $[-\ell,\ell]$ to
bound the integral of $f$ gives
\begin{equation}
    \int_{\cC_{-\ell,\ell}}
    |  u_{\theta}|^2  \leq 2 \ell \, \max_{[-\ell,\ell]} \, f  <  C \, \epsilon_2 \,
     \int_{\cC_{-2\ell,2\ell}}  |\nabla u|^2 \, .
\end{equation}
The proposition follows by choosing $\epsilon_2 > 0$ so that $C \,
\epsilon_2 < \min \{ 1/4 , \,  \delta \}$ and then choosing $\ell$
so that \eqr{e:choosen} holds.
\end{proof}

\subsection{Weak compactness of almost harmonic maps}

We will need a compactness theorem for a sequence of maps  $u^j$
in $W^{1,2} (\SS^2 , M)$ which have uniformly bounded energy and
are locally well-approximated by harmonic maps. Before stating
this precisely, it is useful to recall the situation for harmonic
maps.  Suppose therefore that $u^j:\SS^2 \to M$ is a sequence of
harmonic maps with $\Energy (u^j) \leq E_0$ for some fixed $E_0$.
 After passing to a subsequence, we can assume that the measures
 $|\nabla u^j|^2 \, dx$ converge and there is
 a finite set $\cS$ of points where the energy concentrates so
that:
\begin{align}    \label{e:smalle1}
    {\text{If }} x &\in \cS , {\text{ then }} \inf_{r > 0} \, \left[ \lim_{j\to \infty} \, \int_{B_r(x)}
    |\nabla u^j|^2 \right] \geq \epsilon_{SU} \, . \\
{\text{If }} x &\notin \cS , {\text{ then }}
    \inf_{r > 0} \, \left[ \lim_{j\to \infty} \, \int_{B_r(x)}
    |\nabla u^j|^2 \right] < \epsilon_{SU} \, . \label{e:smalle2}
\end{align}
The constant $\epsilon_{SU} > 0$ comes from \cite{SaU}, so that
\eqr{e:smalle2} implies uniform $C^{2,\alpha}$ estimates on the
$u^j$'s in some neighborhood of $x$.  Hence, Arzela-Ascoli and a
diagonal argument give a further subsequence of the $u^j$'s
$C^2$-converging   to a harmonic map on every compact subset of
$\SS^2 \setminus \cS$.  We will need a more general version of
this, where $u^j:\SS^2 \to M$ is a sequence of $W^{1,2}$ maps with
$\Energy(u^j) \leq E_0$ that are $\epsilon_0$-almost harmonic in
the following sense:

\noindent ($B_0$) If $B \subset \SS^2$ is any ball with $\int_B
|\nabla u^j|^2 < \epsilon_0$, then there is an
 energy minimizing
map $v:   \frac{1}{8}B \to M$ with the same boundary values as
$u^j$ on $\partial \frac{1}{8}B$ with
\begin{equation}    \label{e:lfv3}
    \int_{\frac{1}{8}B} \left| \nabla u^j - \nabla v \right|^2 \leq  1/j \,
    .  \notag
\end{equation}

 \begin{Lem}     \label{l:gl2}
Let   $\epsilon_0 > 0$ be  less than $\epsilon_{SU}$. If
  $u^j:\SS^2 \to M$ is a sequence of $W^{1,2}$ maps satisfying ($B_0$)
  and with
$\Energy(u^j) \leq E_0$, then there exists a finite collection of
points $\{ x_1 , \dots , \, x_k \}$,  a subsequence still denoted
by $u^j$, and a harmonic map $u:\SS^2 \to M$ so that $u^j \to u$
weakly in $W^{1,2}$ and if $K \subset \SS^2 \setminus \{ x_1 ,
\dots , \, x_k \}$ is compact, then $u^j \to u$ in $W^{1,2}(K)$.
Furthermore, the measures $|\nabla u^j|^2 \, dx$ converge to a
 measure $\nu$ with $\epsilon_{0} \leq \nu (x_i)$ and $\nu
(\SS^2) \leq E_0$.
\end{Lem}

\begin{proof}
After passing to a subsequence, we can assume that:
\begin{itemize}
\item The $u^j$'s converge weakly in $W^{1,2}$
 to a $W^{1,2}$ map $u:\SS^2 \to M$.
\item
 The measures
$|\nabla u^j|^2 \, dx$ converge to a limiting measure $\nu$ with
$\nu (\SS^2) \leq E_0$.
\end{itemize}
  It follows that there are at
most $E_0/\epsilon_0$ points $x_1 , \dots , x_k$ with $ \lim_{r
\to 0} \, \nu \, \left( B_{r} (x_j) \right) \geq
    \epsilon_0 \, .
$

We will show next that away from the $x_i$'s the convergence is
strong in $W^{1,2}$ and $u$ is harmonic.  To see this, consider a
point  $x \notin \{ x_1 , \dots , \, x_k \}$.  By definition,
there exist $r_x
> 0$ and $J_x$ so that $\int_{B_{r_x}(x)} |\nabla u^j|^2 <
\epsilon_0$ for $j \geq J_x$.   In particular, ($B_0$) applies so
we get energy minimizing maps $v^j_x:  \frac{1}{8} \, B_{r_x}(x)
\to M$ that agree with $u^j$ on $\partial \frac{1}{8} \,
B_{r_x}(x)$ and satisfy
\begin{equation}
    \int_{ \frac{1}{8} \, B_{r_x}(x) } \, \left| \nabla
v^j_x - \nabla u^j  \right|^2 \leq 1 / j \, . \label{e:useH2}
\end{equation}
(Here $\frac{1}{8} \, B_{r_x}(x)$ is the ball in $\SS^2$ centered
at $x$ so that the stereographic projection $\Pi_x$ which takes
$x$ to $0\in \RR^2$ takes $\frac{1}{8} \, B_{r_x}(x)$ and
$B_{r_x}(x)$ to balls centered at $0$ whose radii differ by a
factor of $8$.)
 Since $\Energy
(v^j_x) \leq \epsilon_0 \leq \epsilon_{SU}$, it follows from lemma
 $3.4$ in  \cite{SaU} (cf. \cite{SU1}) that a subsequence
of the $v^j_x$'s converges strongly  in  $W^{1,2}(\frac{1}{9}
\,B_{r_x}(x))$ to
 a harmonic map $v_x:
\frac{1}{9} \,B_{r_x}(x) \to M$. Combining with the triangle
inequality and \eqr{e:useH2}, we get
\begin{equation}    \label{e:vx1}
    \int_{ \frac{1}{9} \, B_{r_x}(x) } \left| \nabla u^j - \nabla v_x \right|^2
    \leq 2 \,  \int_{\frac{1}{9} \, B_{r_x}(x) } \left| \nabla u^j - \nabla
    v_x^j
    \right|^2 + 2 \,  \int_{ \frac{1}{9} \, B_{r_x}(x) } \left| \nabla v^j_x - \nabla v_x
    \right|^2 \to 0 \, .
\end{equation}
Similarly, this convergence, the triangle inequality,
\eqr{e:useH2}, and the Dirichlet Poincar\'e inequality (theorem
$3$ on page $265$ of \cite{E}; this applies since $v^j_x$ equals
$u^j$ on $\partial \frac{1}{8} \,B_{r_x}(x)$) give
\begin{equation}    \label{e:vx2}
    \int_{\frac{1}{9} \, B_{r_x}(x) } \left|   u^j -   v_x \right|^2
    \leq 2 \,  \int_{\frac{1}{8} \, B_{r_x}(x) } \left|   u^j -
    v_x^j
    \right|^2 + 2 \,  \int_{\frac{1}{9} \, B_{r_x}(x) } \left|   v^j_x -   v_x
    \right|^2 \to 0 \, .
\end{equation}
Combining \eqr{e:vx1} and \eqr{e:vx2}, we see that the $u^j$'s
converge to $v_x$ strongly in $W^{1,2}(\frac{1}{9} \,B_{r_x}(x))$.
In particular, $u \big|_{\frac{1}{9} \, B_{r_x}(x) } = v_x$. We
conclude that $u$ is harmonic on $\SS^2 \setminus \{ x_1 , \dots ,
\, x_k \}$. Furthermore, since any compact $K \subset \SS^2
\setminus \{ x_1 , \dots , \, x_k \}$ can be covered by a finite
number of such ninth-balls, we get that $u^j \to u$ strongly in
$W^{1,2}(K)$.

Finally, since $u$   has finite energy, it must have removable
singularities  at each of the $x_i$'s and, hence, $u$
 extends to a harmonic map on all of $\SS^2$ (see theorem
 $3.6$ in
 \cite{SaU}).
\end{proof}

\subsection{Almost harmonic maps on cylinders}

The main result of this subsection, Proposition \ref{c:c1} below,
extends Proposition \ref{c:ann2} from harmonic maps     to
``almost harmonic'' maps.  Here ``almost harmonic'' is made
precise in Definition \ref{d:ah} below and roughly means that
harmonic replacement on certain balls does not reduce the energy
  by much.

\begin{Def}     \label{d:ah}
Given $\nu > 0$ and a cylinder $\cC_{r_1, r_2}$,   we will say
that a $W^{1,2}(\cC_{r_1, r_2} , \, M)$ map $u$ is $\nu$-almost
harmonic if for any  finite collection of disjoint closed balls
$\cB$ in
  the conformally equivalent annulus $B_{\e^{r_2}} \setminus B_{ \e^{r_1}} \subset \RR^2$  there is an
 energy minimizing
map $v:   \cup_{\cB} \frac{1}{8}B \to M$ that equals $u$ on $
\cup_{\cB} \frac{1}{8} \partial B $ and satisfies
\begin{equation}
    \int_{\cup_{\cB} \frac{1}{8}B} \left| \nabla u - \nabla v \right|^2 \leq  \frac{ \nu}{2} \, \int_{\cC_{r_1, r_2}} |\nabla u|^2 \, .
\end{equation}
\end{Def}

  We have used a slight abuse of notation, since our sets will always be
  thought of as being subsets of the cylinder; i.e., we identify   Euclidean
  balls in the annulus with their image under the conformal  map
  to the cylinder.

In this subsection and the two that follow it, given $\delta > 0$,
the constants $\ell \geq 1$  and $\epsilon_2 > 0$ will be given by
Proposition \ref{c:ann2};
 these depend only on $M$ and $\delta$.

\begin{Pro}     \label{c:c1}
Given $\delta > 0$, there exists $\nu > 0$ (depending on $\delta$
and $M$) so that if  $m$ is a positive integer and $u$ is
$\nu$-almost harmonic from $\cC_{-(m+3)\ell,3\ell}$ to $M$ with
$\Energy (u) \leq \epsilon_{2}$, then
\begin{equation}
    \int_{ \cC_{-m\ell,0} } \left|  u_{\theta} \right|^2 \leq 7 \,\delta \,
    \int_{\cC_{-(m+3)\ell,3\ell}} |\nabla u|^2
        \, .
\end{equation}
\end{Pro}

We will prove Proposition \ref{c:c1} by using a compactness
argument to reduce it to the case of harmonic maps and then appeal
to Proposition \ref{c:ann2}.  A key difficulty is that there is no
upper bound on the length of the cylinder in Proposition
\ref{c:c1} (i.e., no upper bound on $m$), so we cannot directly
apply the compactness argument. This will be taken care of by
dividing the cylinder into subcylinders of a fixed size and then
using a covering argument.

\subsection{The compactness argument}

The next lemma extends  Proposition \ref{c:ann2} from harmonic
maps on   $\cC_{-3\ell,3\ell}$ to almost harmonic maps. The main
difference from Proposition \ref{c:c1} is that the cylinder is of
a fixed size in Lemma \ref{p:perturb}.

\begin{Lem}     \label{p:perturb}
Given $\delta > 0$, there exists $\mu > 0$ (depending on $\delta$
and $M$) so that if    $u$ is a $\mu$-almost harmonic  map from
$\cC_{-3\ell,3\ell}$ to $M$ with $\Energy (u) \leq \epsilon_{2}$,
then
\begin{equation}
    \int_{ \cC_{-\ell,\ell} } |u_{\theta}|^2 \leq \delta \,
    \int_{\cC_{-3\ell,3\ell}}
    |\nabla u|^2
        \, .
\end{equation}
\end{Lem}

\begin{proof}
We will argue by contradiction, so suppose that there exists a
sequence $u^j$ of $1/j$-almost harmonic   maps
 from $\cC_{-3\ell,3\ell}$ to $M$ with $\Energy (u^j) \leq
\epsilon_{2}$ and
\begin{equation}    \label{e:badj}
    \int_{ \cC_{-\ell,\ell} } |u^j_{\theta}|^2 > \delta \, \int_{\cC_{-3\ell,3\ell}}
    |\nabla u^j|^2
        \, .
\end{equation}
 We will show that a subsequence of the $u^j$'s converges to a
non-constant harmonic map that contradicts Proposition
\ref{c:ann2}.  We will consider two separate cases,
depending on whether or not $\Energy (u^j)$ goes to $0$.

Suppose first that $\limsup_{j\to \infty} \Energy (u^j) > 0$.
 The upper bound on the energy combined with being
$1/j$-almost harmonic (and the compactness of $M$) allows us to
argue as in Lemma \ref{l:gl2} to get a subsequence that
converges in $W^{1,2}$ on compact subsets of $\cC_{-3\ell,3\ell}$
to a {\underline{non-constant}} harmonic map $\tilde{u}:\cC_{-3\ell,3\ell} \to M$.
Furthermore, using the $W^{1,2}$ convergence on
$\cC_{-\ell,\ell}$ together with the lower semi-continuity of
energy, \eqr{e:badj} implies that
$\int_{ \cC_{-\ell,\ell} } |\tilde{u}_{\theta}|^2 \geq \delta \, \int_{\cC_{-3\ell,3\ell}}
    |\nabla \tilde{u}|^2$.  This contradicts Proposition
\ref{c:ann2}.

Suppose now that   $\Energy (u^j) \to 0$.
        Replacing $u^j$ by $v^j = (u^j - u^j(0))/ (\Energy (u^j))^{1/2}$ gives
 a sequence of maps to $M_j = (M - u^j(0))/\Energy (u^j))^{1/2}$  with $\Energy (v^j)=1$  and, by  \eqr{e:badj},
 $\int_{\cC_{-\ell,\ell}} |v^j_{\theta}|^2 > \delta > 0$.
 Furthermore, the $v^j$'s are also $1/j$-almost harmonic (this property is invariant under dilation), so we can
still argue as in Lemma \ref{l:gl2} to get a subsequence that converges in  $W^{1,2}$ on compact subsets of
$\cC_{-3\ell,3\ell}$ to a harmonic map $v: \SS^2 \to \RR^N$ (we are using here that a subsequence of the $M_j$'s converges to an affine space).  As before,
\eqr{e:badj} implies that
$\int_{ \cC_{-\ell,\ell} } |v_{\theta}|^2 \geq \delta \, \int_{\cC_{-3\ell,3\ell}}
    |\nabla v|^2$. This time our normalization gives $\int_{\cC_{-\ell,\ell}} |v_{\theta}|^2 \geq \delta$ so that
$v$ contradicts Proposition
\ref{c:ann2} (see Remark \ref{r:harmonic}), completing the proof.
\end{proof}

\subsection{The proof of
Proposition \ref{c:c1}}

\begin{proof}
(of Proposition \ref{c:c1}). For each integer  $j = 0, \dots , m$,
let $ \cC(j) = \cC_{-(j+3)\ell,(3-j)\ell}$ and let
 $\mu > 0$ be given by Lemma \ref{p:perturb}.  We will
say that the $j$-th cylinder $\cC(j)$ is {\emph{good}} if the
restriction of $u$ to $\cC(j)$ is $\mu$-almost harmonic;
otherwise, we will say that $\cC(j)$ is {\emph{bad}}.

On each good $\cC(j)$, we apply Lemma \ref{p:perturb}
  to get
\begin{equation}
    \int_{ \cC_{-(j+1)\ell,(1-j)\ell} } |u_{\theta}|^2 \leq \delta \,
    \int_{\cC(j)} |\nabla u|^2 \, ,
\end{equation}
so that summing this over the good $j$'s gives
\begin{equation}    \label{e:cjgood}
    \sum_{j \, {\text{good}} } \, \int_{ \cC_{-(j+1)\ell,(1-j)\ell} } |u_{\theta}|^2 \leq \delta \,
    \sum_{j \, {\text{good}} } \, \int_{\cC(j)} |\nabla u|^2 \leq 6
    \, \delta \, \int_{\cC_{-(m+3)\ell,3\ell}} |\nabla u|^2
    \, ,
\end{equation}
where the last inequality used that each $\cC_{i, i+1}$ is
contained in at most $6$ of the $\cC(j)$'s.

We will complete the proof by showing that the total energy (not
just the $\theta$-energy) on the bad   $\cC(j)$'s is small. By
definition, for each bad   $\cC(j)$, we can choose a finite
 collection of disjoint closed balls $\cB_j$ in $\cC(j)$ so
that if $v: \frac{1}{8} \cB_j \to M$ is any energy-minimizing map
that equals  $u$ on $\partial \frac{1}{8} \cB_j$, then
\begin{equation}    \label{e:defbad}
      \int_{\frac{1}{8} \cB_j} |\nabla u- \nabla v|^2 \geq a_j > \mu \, \int_{   \cC(j) } |\nabla u|^2 \, .
\end{equation}
  Since the interior of each $\cC(j)$ intersects only the
$\cC(k)$'s with $0 < |j-k| \leq 5$, we can  divide the bad
$\cC(j)$'s into ten subcollections so that the interiors of the
$\cC(j)$'s in each subcollection are pair-wise disjoint. In
particular, one of these disjoint subcollections, call it
$\Gamma$, satisfies
\begin{equation}    \label{e:1item}
    \sum_{j \in \Gamma} \, a_j \geq \frac{1}{10} \sum_{j \, {\text{bad}}} \,
    a_j
    \geq \frac{1}{10} \sum_{j \, {\text{bad}}}
      \mu  \, \int_{   \cC(j) } |\nabla u|^2  \, ,
\end{equation}
where the last inequality used  \eqr{e:defbad}.

However, since  $\cup_{j \in \Gamma} \, \cB_j$ is itself a finite
collection of disjoint closed balls in the entire cylinder
$\cC_{-(m+3)\ell,3\ell}$ and
 $u$ is $\nu$-almost
harmonic on $\cC_{-(m+3)\ell,3\ell}$, we get that
\begin{equation}    \label{e:controlbad}
    \frac{\mu}{10 }  \, \sum_{j \, {\text{bad}}}
     \, \int_{   \cC(j) } |\nabla u|^2  \leq  \nu \, \int_{\cC_{-(m+3)\ell,3\ell}} |\nabla u|^2 \, .
\end{equation}
To get the proposition, combine \eqr{e:cjgood} with
\eqr{e:controlbad}  to get
\begin{equation}
      \int_{\cC_{-m\ell,0}} |u_{\theta}|^2
     \leq    \left( 6
    \, \delta   +
\frac{10   \,   \nu}{\mu} \right)
    \, \int_{\cC_{-(m+3)\ell,3\ell}} |\nabla u|^2
    \, .
\end{equation}
Finally, choosing $\nu$ sufficiently small completes the proof.
\end{proof}

\subsection{Bubble compactness}

We will now prove Proposition \ref{p:gl2} using a variation of the
renormalization procedure developed in \cite{PaW} for
pseudo-holomorphic curves and later used in \cite{Pa} for harmonic
maps.  A key point in the proof will be that the uniform energy
bound, (A), and (B) are all dilation invariant, so they apply also
to the compositions of the $u^j$'s with any sequence of conformal
dilations of $\SS^2$.

\begin{proof}
(of Proposition \ref{p:gl2}).   We will use the energy bound and
(B) to show that a subsequence of the $u^j$'s converges in the
sense of (B1), (B2), and (B3) of Definition \ref{d:bubble} to a
collection of harmonic maps. We will then come back and use (A)
and (B) to show that the energy equality (B4) also holds. Hence,
the subsequence bubble converges and, thus by Proposition
\ref{l:bubvar}, also varifold converges.

Set $\delta = 1/21$ and let   $\ell \geq 1$ and $\epsilon_2 > 0$ be
given by Proposition
 \ref{c:ann2}.  Set $\epsilon_3 = \min \{  \epsilon_{0}/2 , \epsilon_2 \}$.

{\underline{Step 1: Initial compactness}}. Lemma \ref{l:gl2} gives
a finite collection of singular points $\cS_0 \subset \SS^2$,  a
harmonic map $v_0 : \SS^2 \to M$, and a subsequence (still denoted
$u^j$) that
  converges to $v_0$ weakly in $W^{1,2}(\SS^2)$ and
 strongly in $W^{1,2}(K)$ for any compact subset $K
\subset \SS^2 \setminus \cS_0$.  Furthermore, the measures
$|\nabla u^j|^2 \, dx$ converge to a measure $\nu_0$ with $\nu_0
(\SS^2) \leq E_0$ and each singular point in $x \in \cS_0$ has
$\nu_0 (x) \geq \epsilon_0$.

{\underline{Step 2: Renormalizing at a singular point}}. Suppose
  that $x \in \cS_0$ is a singular point from the
first step. Fix a radius $\rho > 0$ so that $x$ is the only
singular point in $B_{2\rho}(x)$ and $\int_{B_{\rho}(x)} |\nabla
v_0|^2 < \epsilon_3 / 3$.
   For each
$j$,  let $r_j >0$ be the smallest radius so that
\begin{equation}
    \inf_{y \in B_{\rho - r_j} (x)} \, \, \int_{ B_{\rho}(x) \setminus B_{r_j}(y) } \, |\nabla u^j|^2
    = \epsilon_3 \, ,
\end{equation}
and choose   a  ball $B_{r_j}(y_j) \subset B_{\rho}(x)$   with
$\int_{ B_{\rho}(x) \setminus B_{r_j}(y_j) } \, |\nabla u^j|^2
    = \epsilon_3$.
  Since the $u^j$'s converge to
    $v_0$ on compact subsets of $B_{\rho}(x) \setminus \{
    x \}$, we get that
    $y_j \to x$ and $r_j \to 0$.
   For each $j$, let $\Psi_j:\RR^2 \to \RR^2$ be the
    ``dilation'' that takes $B_{r_j}(y_j)$
    to the unit ball $B_{1}(0) \subset \RR^2$.   By dilation invariance,    the
    dilated maps
$\tilde{u}_1^j   =    u^j \circ \Psi_j^{-1}$ still satisfy (B) and
have the same energy.  Hence, Lemma \ref{l:gl2} gives a
subsequence (still denoted by $\tilde{u}_1^j $), a finite singular
set $\cS_1$,   and a harmonic map $v_1$ so that the
 $\tilde{u}_1^j \circ \Pi$'s converge to $v_1$ weakly in $W^{1,2}(\SS^2)$ and
 strongly in $W^{1,2}(K)$ for any compact subset $K
\subset \SS^2 \setminus \cS_1$.  Moreover, the measures $|\nabla
\tilde{u}_1^j \circ \Pi|^2 \, dx$'s converge to a measure $\nu_1$.

The choice of the balls $B_{r_j}(y_j)$ guarantees that $\nu_1 (
\SS^2 \setminus \{ p^+ \}) \leq \nu_0 (x)$ and $\nu_1 (S^-) \leq
\nu_0 (x) - \epsilon_3$. (Recall that stereographic projection
$\Pi$ takes the open southern hemisphere $S^-$ to the open unit
ball in $\RR^2$.) The key point for iterating this is the
following claim:
\begin{enumerate}
\item[($\star$)] The maximal energy concentration at any $y \in
\cS_1 \setminus \{ p^+ \}$ is at most $\nu_0 (x) - \epsilon_3/3$.
\end{enumerate}
Since the energy at a singular point or the energy for a
non-trivial harmonic map is at least $\epsilon_0 > \epsilon_3$,
the only one way that ($\star$) could possibly fail is if $v_1$ is
constant, $\cS_1$ is exactly two points $p^+$ and $y$, and at most
$\epsilon_3/3$ of $\nu_0 (x)$ escapes at $p^+$.
 However, this would imply that all but at most $2\epsilon_3/3$ of
 the $\int_{B_{\rho}(x)} |\nabla u^j|^2$ is in $B_{t_j}(y_j)$ with
 $\frac{t_j}{r_j} \to 0$ which contradicts the minimality of
 $r_j$.

{\underline{Step 3: Repeating this}}.   We repeat this blowing up
construction at the remaining singular points in $\cS_0$, as well
as each of the singular points $\cS_1$ in the southern hemisphere,
  etc., to get new limiting harmonic maps and new
singular points to blow up at.  It follows from ($\star$) that
  this must terminate after at most $3\,E_0/\epsilon_3$ steps.

{\underline{Step 4: The necks}}.  We have shown that the $u^j$'s
converge to a collection of harmonic maps in the sense of (B1),
(B2), and (B3).  It remains to show (B4), i.e., that
  the $v_k$'s  accounted for all of the
energy in the sequence $u^j$ and no energy was lost in the limit.

To understand how energy could be lost, it is useful to re-examine
what happens to the energy during the blow up process. At each
stage in the blow up process, energy is ``taken from'' a singular
point $x$ and  then goes to one of two places:
\begin{itemize}
\item It can show up in the new limiting harmonic map of to a
singular point in $\SS^2 \setminus \{ p^+ \}$. \item It can
disappear at the north pole $p^+$ (i.e., $\nu_1 (\SS^2 \setminus
\{ p^+ \}) < \nu_0 (x)$).
\end{itemize}
In the first case, the energy is accounted for in the limit or
survives to a later stage.  However, in the second case, the
energy is lost for good, so this is what we must rule out.

We will argue by contradiction, so suppose that $\nu_1 (\SS^2
\setminus \{ p^+ \}) < \nu_0 (x) - \hat{\delta}$ for some
$\hat{\delta} > 0$.  (Note that we must have $\hat{\delta} \leq
\epsilon_3$.)
   Using the notation in Step 1, suppose therefore
that $A_j= B_{s_j}(y_j)\setminus B_{t_j}(y_j)$ are annuli with:
\begin{equation}    \label{e:N1}
   s_j \to 0 \, , \,  \frac{t_j}{r_j} \to \infty \, , {\text{ and }}
    \int_{A_j} |\nabla u_j|^2 \geq \hat{\delta} > 0 \, .
\end{equation}
 There is obviously
quite a bit of freedom in choosing $s_j$ and $t_j$.  In
particular, we can choose a sequence $\lambda_j \to \infty$ so
that the annuli $\tilde{A}_j = B_{\rho/2}(y_j)\setminus
B_{t_j/\lambda_j}(y_j)$ also satisfies this, i.e., $\lambda_j \,
s_j \to 0$ and $t_j / (\lambda_j r_j) \to \infty$.  It follows
from \eqr{e:N1}
 and the definition of the
$r_j$'s that $\int_{\tilde{A}_j} |\nabla u^j|^2 \leq \epsilon_3
\leq \epsilon_2$. However, combining this with Proposition
\ref{c:c1} (with $\delta = 1/21$) shows that the area must be
strictly less than the energy for $j$ large, contradicting (A),
and  thus completing the proof.
\end{proof}

\section{The proof of Theorem \ref{l:trivmap}}  \label{s:appB}

\subsection{An application of the Wente lemma}

The proof of Theorem \ref{l:trivmap} will use the following $L^2$
estimate for $h \, \zeta $ where $\zeta$ is a $L^2 (B_1)$
holomorphic function and $h$ is a $W^{1,2}$ function vanishing on
$\partial B_1$.

\begin{Pro}     \label{c:hardy}
If $\zeta$ is a holomorphic function on $B_1 \subset \RR^2$ and $h
\in  W^{1,2}_0 (B_1)$, then
\begin{equation}    \label{e:1a}
    \int_{B_1} h^2 \, |\zeta|^2 \leq 8 \, \left( \int_{B_1}
    |\nabla h|^2 \right) \, \left( \int_{B_1}
    |\zeta|^2 \right) \, .
\end{equation}
\end{Pro}

The estimate \eqr{e:1a} does not follow from the Sobolev embedding
theorem as  the product of functions in $L^2$ and $W^{1,2}$ is in
$L^p$ for $p< 2$, but not necessarily for $p=2$. To get around
this, we will use the following lemma of H. Wente (see  \cite{W};
cf. theorem $3.1.2$ in \cite{He1}).

\begin{Lem}     \label{l:wente}
If $B_1 \subset \RR^2$ and $u,v \in W^{1,2} (B_1)$, then there
exists  $\phi \in C^0 \cap W^{1,2}_0(B_1)$ with $\Delta \phi =
\langle (\partial_{x_1} \, u, \partial_{x_2} \,  u) , (-
\partial_{x_2} \,v ,  \partial_{x_1} \, v) \rangle$ so that
\begin{equation}
    ||\phi||_{C^0} + ||\nabla \phi ||_{L^2} \leq   ||\nabla u||_{L^2} \,
        ||\nabla v||_{L^2} \, .
\end{equation}
\end{Lem}

\begin{proof}
(of Proposition \ref{c:hardy}.) Let $f$ and $g$ be the real and
imaginary parts, respectively, of the holomorphic function
$\zeta$, so that the Cauchy-Riemann equations give
\begin{equation}    \label{e:cr2}
    \partial_{x_1} f = \partial_{x_2} g {\text{ and }} \partial_{x_2} f = - \partial_{x_1} g \, .
\end{equation}
Since $B_1$ is simply connected, \eqr{e:cr2} gives  functions $u$
and $v$ on $B_1$ with $\nabla u = (g,f)$ and $\nabla v = (f,-g)$.
We have
\begin{equation}    \label{e:ns}
    |\nabla u|^2 = |\nabla v|^2 = \langle (\partial_{x_1} u, \partial_{x_2} u) , (- \partial_{x_2} v ,
\partial_{x_1} v) \rangle = |\zeta|^2 \, .
\end{equation}
Therefore, Lemma \ref{l:wente} gives $\phi$ with $\Delta \phi =
|\zeta|^2$, $\phi |_{\partial B_1} = 0$, and
\begin{equation}    \label{e:phib}
    ||\phi||_{C^0} + ||\nabla \phi ||_{L^2} \leq   \int |\zeta|^2  \, .
\end{equation}
Applying Stokes' theorem to $\dv (h^2 \nabla \phi)$ and using
Cauchy-Schwarz gives
\begin{equation}    \label{e:st1}
    \int h^2 |\zeta|^2 = \int h^2 \Delta \phi \leq \int |\nabla
    h^2| \, |\nabla \phi| \leq 2 \,
    ||\nabla h||_{L^2} \, \left( \int h^2 \, |\nabla \phi|^2 \right)^{1/2} \, .
\end{equation}
Applying Stokes' theorem to $\dv (h^2 \phi \nabla \phi)$, noting
that $\Delta \phi \geq 0$, and using \eqr{e:st1} gives
\begin{equation}    \label{e:st2}
    \int h^2 |\nabla \phi |^2 \leq \int   |\phi| \left(  h^2 \,  \Delta \phi  +
      |\nabla h^2| \, |\nabla \phi| \right) \leq 4 \, ||\phi||_{C^0}
      \, ||\nabla h||_{L^2} \,
        \left( \int h^2 \, |\nabla \phi|^2 \right)^{1/2}
 \, ,
\end{equation}
so that     $\left( \int h^2 |\nabla \phi |^2 \right)^{1/2} \leq 4
\, ||\nabla h||_{L^2} \, ||\phi||_{C^0}$. Finally, substituting
this bound back into \eqr{e:st1} and using \eqr{e:phib} to bound
$||\phi||_{C^0}$ gives the proposition.
\end{proof}

\subsection{An application to harmonic maps}  \label{s:he}

\begin{Pro}     \label{p:harmhardy}
Suppose that $M\subset \RR^N$ is a smooth closed isometrically
embedded manifold.  There exists a constant $\epsilon_0 > 0$
(depending on $M$) so that if
  $v:B_1 \to M$ is a $W^{1,2}$ weakly harmonic map with
energy at most $ \epsilon_0$,  then $v$ is
  a smooth harmonic map.
In addition, for any $h \in W_0^{1,2}(B_1)$, we have
\begin{equation}    \label{e:1b}
    \int_{B_1} |h|^2 \, |\nabla v|^2 \leq C \, \left( \int_{B_1}
    |\nabla h|^2 \right) \, \left( \int_{B_1}
    |\nabla v|^2 \right) \, .
\end{equation}
\end{Pro}

\begin{proof}
The first claim follows immediately from F. H\'elein's 1991
regularity theorem for weakly harmonic maps from surfaces; see
\cite{He2} or theorem $4.1.1$ in \cite{He1}.

We will show that   \eqr{e:1b} follows by combining estimates from
the proof of theorem $4.1.1$ in
\cite{He1}{\footnote{Alternatively, one could use the recent
results of T. Rivi\`ere, \cite{Ri}.}} with Proposition
\ref{c:hardy}. Following \cite{He1}, we can assume that   the
pull-back $v^{\ast}(TM)$ of the tangent bundle of $M$ has
orthonormal frames on $B_1$ and, moreover, that there is a finite
energy harmonic section $e_1 , \dots , e_n$   of the bundle of
orthonormal frames for $v^{\ast}(TM)$ (the frame $e_1 , \dots ,
e_n$ is usually called a {\emph{Coulomb gauge}}). Set $\alpha^j =
\langle
\partial_{x_1} v , e_j \rangle - i \, \langle
\partial_{x_2} v , e_j
    \rangle$
for $j= 1 , \dots , n$.  Since $e_1 , \dots , e_n$ is an
orthonormal frame for $v^{\ast}(TM)$,  we have
\begin{equation}    \label{e:alp}
    |\nabla v|^2 = \sum_{j=1}^n |\alpha^j|^2 \, .
\end{equation}
On pages 181 and 182 of \cite{He1}, H\'elein uses that the frame
$e_1 , \dots , e_n$ is harmonic to construct an $n\times n$
matrix-valued function $\beta$ (i.e., a map
    $\beta : B_1 \to GL(n,\CC)$) with $|\beta| \leq C$, $|\beta^{-1}| \leq C$, and with
$
    \partial_{\bar{z}} \, \left( \beta^{-1} \, \alpha \right) = 0
$ (where the constant $C$ depends only on $M$ and the bound for
the energy of $v$; see also lemma $3$ on page $461$ in \cite{Q}
where this is also stated). In particular, we get an $n$-tuple of
holomorphic functions $(\zeta^1 , \dots , \zeta^n) = \zeta =
\beta^{-1} \, \alpha$, so that
\begin{equation}    \label{e:holo2}
     C^{-2} \,|\zeta|^2 \leq  |\alpha|^2   =  |\beta \, \zeta|^2 \leq C^2 \, |\zeta|^2
    \, .
\end{equation}
 The claim \eqr{e:1b} now follows from Proposition \ref{c:hardy}.
 Namely, using \eqr{e:alp}, the second inequality in
 \eqr{e:holo2}, and then applying Proposition \ref{c:hardy} to the
 $n$ holomorphic functions $\zeta^1 , \dots , \zeta^n$ gives
\begin{equation}
    \int  |h|^2 \, |\nabla v|^2
    \leq C^2 \, \int |h|^2 \, |\zeta|^2
    \leq 8 \, C^2 \,  \int
    |\nabla h|^2  \,   \int
    |\zeta|^2
    \leq 8 \, C^4 \,   \int
    |\nabla h|^2   \,   \int
    |\nabla v|^2   \, ,
\end{equation}
where the last inequality used the first inequality in
\eqr{e:holo2} and \eqr{e:alp}.
\end{proof}

\subsection{The proof of Theorem \ref{l:trivmap}}  \label{s:B}

\begin{proof}
(of Theorem \ref{l:trivmap}.)
 Use Stokes' theorem and   that $u$ and $v$
are equal on $\partial B_1$ to get
\begin{equation}
    \int  |\nabla u|^2 - \int  |\nabla v|^2 -
     \int  \left| \nabla (u - v) \right|^2
    =   - 2 \, \int   \langle (u - v)  ,  \Delta v \rangle \equiv \Psi \, .
\end{equation}
To show \eqr{e:trivmap}, it suffices to bound $|\Psi|$ by
$\frac{1}{2} \,
        \int  \left| \nabla v - \nabla u \right|^2$.

 The harmonic
map equation \eqr{e:footnote} implies that $\Delta v$ is
perpendicular to $M$ and
\begin{equation}    \label{e:hmeq}
    |\Delta v| \leq  |\nabla v|^2 \, \sup_M |A|    \, .
\end{equation}

We will need the elementary geometric fact that there exists a
constant $C$ depending on $M$ so that whenever $x,y \in M$, then
\begin{equation}    \label{e:almosttan}
    \left| (x-y)^N \right| \leq C \, |x-y|^2 \, ,
\end{equation}
where $(x-y)^N$ denotes the normal part of the vector $(x-y)$ at
the point $x \in M$ (the same bound holds at $y$ by symmetry). The
point is that either $|x-y| \geq 1/C$   so \eqr{e:almosttan} holds
trivially or the vector $(x-y)$ is ``almost tangent'' to $M$.

Using that $u$ and $v$ both map to   $M$, we can apply
\eqr{e:almosttan} to get
$\left| (u-v)^N \right| \leq C \, |u-v|^2$,
where the normal projection is at the point $v(x) \in M$.
 Putting all of this
together gives
\begin{equation}    \label{e:gotpsi}
    \left| \Psi \right| \leq C \, \int  |v-u|^2 \, |\nabla
    v|^2 \, ,
\end{equation}
where $C$ depends on $\sup_M |A|$.  As long as $\epsilon_1$ is
less than $\epsilon_0$, we can apply Proposition \ref{p:harmhardy}
with $h=|u-v|$ to get
\begin{equation}    \label{e:fromhardy}
    \int  |v-u|^2 \, |\nabla
    v|^2 \leq C' \, \left( \int
    |\nabla |u-v||^2 \right) \, \left( \int
    |\nabla v|^2 \right) \leq C' \, \epsilon_1 \, \int
    |\nabla u- \nabla v|^2
    \, .
\end{equation}
The lemma follows by combining \eqr{e:gotpsi} and
\eqr{e:fromhardy} and then taking $\epsilon_1$ sufficiently small.
\end{proof}

Combining Corollary \ref{c:trivmap1} and the regularity theory of
\cite{Mo1}, or \cite{SU1}, for energy minimizing maps   recovers
H\'elein's theorem that weakly harmonic maps from surfaces are
smooth.  Note, however, that we used estimates from \cite{He1} in
the proof of Theorem \ref{l:trivmap}.

\section{The equivalence of energy and area}    \label{s:eq}

By \eqr{e:tr1},    Proposition \ref{p:eq} follows once we show
that $W_E \leq W_A$. The corresponding result for the Plateau
problem is proven by taking a minimizing sequence for area and
reparametrizing   to make these maps conformal, i.e., choosing
isothermal coordinates. There are a few technical difficulties in
carrying this out since the pull-back metric may be degenerate and
is only in $L^1$, while the existence of isothermal coordinates
requires that the induced metric be positive and bounded; see,
e.g., proposition $5.4$ in \cite{SW}. We will follow the same
approach here, the difference is that we need  the
reparametrizations to vary continuously with $t$.

\subsection{Density of smooth mappings}

The next lemma   observes that the regularization using
convolution of Schoen-Uhlenbeck in the proposition in section $4$
of \cite{SU2} is continuous.

\begin{Lem}     \label{l:density}
Given $\gamma \in \Omega$ and $\epsilon > 0$, there exists a
regularization $\tilde{\gamma} \in \Omega_{\gamma}$ so that
\begin{equation}    \label{e:dense}
    \max_t \, \, || \tilde{\gamma} (\cdot , t) -  {\gamma} (\cdot ,
    t)||_{W^{1,2}} \leq \epsilon \, ,
\end{equation}
 each
slice $\tilde{\gamma} (\cdot , t)$ is $ C^{2}$, and the map $t \to
\tilde{\gamma} (\cdot , t)$ is continuous from $[0,1]$ to $C^2
(\SS^2 , M)$.
\end{Lem}

\begin{proof}
Since $M$
is smooth, compact and embedded, there exists a $\delta
> 0$ so that for each $x$ in the $\delta$-tubular neighborhood $M_{\delta}$
of $M$ in $\RR^N$, there is a unique closest
point $\Pi (x) \in M$ and  so the map $x \to \Pi(x)$ is smooth.
  $\Pi$ is called
{\emph{nearest point projection}}.

Given $y$ in the open ball $B_1 (0) \subset \RR^3$, define  $T_y
:\SS^2 \to \SS^2$ by $T_y(x) = \frac{x -y}{|x - y|}$.  Since each
$T_y$ is smooth and these maps depend smoothly on $y$, it follows
that the map $(y,f) \to f \circ T_y$ is continuous from $B_1(0)
\times C^0 \cap W^{1,2}(\SS^2 , \RR^N) \to C^0 \cap W^{1,2}(\SS^2
, \RR^N)$ (this is clear for $f \in C^1$ and follows for $C^0 \cap
W^{1,2}$ by density). Therefore, since $T_0$ is the identity,
given $f \in C^0 \cap W^{1,2}(\SS^2 , \RR^N)$ and $\mu > 0$, there
exists $r>0$ so that $  \sup_{|y| \leq r} \, \, ||f\circ T_y -
f||_{C^0 \cap W^{1,2}} < \mu $. Applying this to $\gamma (\cdot ,
t)$ for each $t$ and using that $t \to \gamma (\cdot , t)$ is
continuous to $C^0 \cap W^{1,2}$ and $[0,1]$ is compact, we get
$\bar{r} > 0$ with
\begin{equation} \label{e:from4}
 \sup_{t \in [0,1]} \, \sup_{|y| \leq \bar{r}}\,  ||T_y \gamma (\cdot   , t)
- \gamma (\cdot , t)||_{C^0 \cap W^{1,2}}
 <
\mu  \, .
\end{equation}
Next fix a smooth radial mollifier $\phi \geq 0$ with integral one
and compact support in the unit ball in $\RR^3$. For each $r \in
(0, 1)$, define $\phi_r (x) = r^{-3} \, \phi (x/r)$ and set
\begin{equation}    \label{e:defgr}
    \gamma_r (x,t) =
\int_{B_r(0)} \phi_r (y) \gamma ( T_y(x) \, ,t) \, dy =
    \int_{B_r(x)} \phi_r (x-y) \gamma ( \frac{y}{|y|} \, ,t) \, dy
    \, .
\end{equation}
We have the following standard properties of convolution with a
mollifier (see, e.g., section $5.3$ and appendix C.$4$ in
\cite{E}):  First, each ${\gamma}_r (\cdot , t)$ is smooth and for
each $k$ the map   $t \to {\gamma}_r (\cdot , t)$ is continuous
from $[0,1]$ to $C^k (\SS^2 , \RR^N)$.  Second,
\begin{align}
 ||{\gamma}_r (\cdot , t) - \gamma (\cdot ,t)||^2_{C^0}
&\leq \sup_{|y| \leq r} \, ||T_y \gamma (\cdot   , t)
- \gamma (\cdot , t)||^2_{C^0} \, ,\label{e:molli}\\
||\nabla {\gamma}_r (\cdot, t) - \nabla \gamma (\cdot ,
t)||^2_{L^2} &\leq \sup_{|y| \leq r} \, ||T_y \gamma (\cdot  , t)
- \gamma (\cdot , t)||^2_{L^2} \, .\notag
\end{align}
It follows from \eqr{e:molli} and \eqr{e:from4} that for $r\leq
\bar{r}$ and all $t$ we have
\begin{equation}    \label{e:unid}
     ||{\gamma}_r (\cdot , t) - \gamma (\cdot , t)||_{C^0 \cap W^{1,2}}
<  \mu \, .
\end{equation}

The map ${\gamma}_r (\cdot , t)$ may not   land in $M$, but it is
in $M_{\delta}$ when $\mu$ is small by \eqr{e:unid}.  Hence, the
map $\tilde{\gamma} (x,t) = \Pi \circ \gamma_r (x,t)
$
satisfies \eqr{e:dense},
 each
slice $\tilde{\gamma} (\cdot , t)$ is $ C^{2}$, and  $t \to
\tilde{\gamma} (\cdot , t)$ is continuous from $[0,1]$ to $C^2
(\SS^2 , M)$. Finally, $s \to \tilde{\gamma}_{sr}$ is an explicit
homotopy connecting $\tilde{\gamma}$ and $\gamma$.
\end{proof}

\subsection{Equivalence of energy and area}

We will also need the existence of isothermal coordinates, taking
special care on the dependence on the metric.  Let
$\SS^2_{g_0}$ denote the round metric on $\SS^2$ with constant
curvature one.

\begin{Lem} \label{l:abe}
Given a $C^{1}$ metric $\tilde{g}$ on $\SS^2$, there is a unique
orientation preserving $C^{1,1/2}$ conformal diffeomorphism
$h_{\tilde{g}}: \SS^2_{g_0} \to \SS^2_{\tilde{g}}$ that fixes $3$
given points.

Moreover, if $\tilde{g}_1$ and $\tilde{g}_2$ are two $C^{1}$
metrics that are both $\geq \epsilon \, g_0$ for some $\epsilon >
0$, then
\begin{equation}    \label{e:contRM}
        ||h_{\tilde{g}_1} - h_{\tilde{g}_2} ||_{C^0 \cap W^{1,2}} \leq C \,
        ||\tilde{g}_1 - \tilde{g}_2||_{C^0} \, ,
\end{equation}
where the constant $C$ depends on $\epsilon$ and the maximum of
the $C^1$ norms of the $\tilde{g}_i$'s.
\end{Lem}

\begin{proof}
The Riemann mapping theorem for variable  metrics (see theorem
$3.1.1$ and corollary $3.1.1$ in \cite{Jo}; cf. \cite{ABe} or
\cite{Mo2}) gives the conformal diffeomorphism $h_{\tilde{g}}:
\SS^2_{g_0} \to \SS^2_{\tilde{g}}$.

We will separately bound the $C^0$ and $W^{1,2}$ norms. First,
 lemma $17$ in \cite{ABe}
gives
\begin{equation}    \label{e:contlinf}
    ||h_{\tilde{g}_1} - h_{\tilde{g}_2} ||_{C^0}  \leq C_1 \,
        ||\tilde{g}_1 - \tilde{g}_2||_{C^0} \, ,
\end{equation}
where $C_1$ depends on $\epsilon$ and the $C^0$ norms of the
metrics.  Second, theorem $8$ in \cite{ABe} gives a uniform $L^p$
bound for $\nabla (h_{\tilde{g}_1} - h_{\tilde{g}_2})$ on any unit
ball in $\SS^2$ where $p>2$ by (8) in \cite{ABe}
\begin{equation}  \label{e:contl2}
    ||\nabla (h_{\tilde{g}_1} - h_{\tilde{g}_2}) ||_{L^p(B_1)}  \leq C_2 \,
        ||\tilde{g}_1 - \tilde{g}_2||_{C^0(\SS^2)} \, ,
\end{equation}
where $C_2$ depends on $\epsilon$ and the $C^0$ norms of the
metrics.  Covering $\SS^2$ by a finite collection of unit balls
and applying H\"older's inequality gives the desired energy bound.
\end{proof}

We can now prove the equivalence of the two widths.

\begin{proof}
(of Proposition \ref{p:eq}). By  \eqr{e:tr1}, we have that $W_A
\leq W_E$.  To prove that
 $W_E \leq W_A$, given
  $\epsilon > 0$, let
$\gamma \in \Omega_{\beta}$ be  a sweepout with $\max_{ t \in [0,
1]} \,  \Area \, (\gamma (\cdot , t )) < W_A + \epsilon/2$.  By
Lemma \ref{l:density}, there is a regularization $\tilde{\gamma}
\in \Omega_{\beta}$ so that each slice $\tilde{\gamma} (\cdot ,
t)$ is $ C^{2}$,  the map $t \to \tilde{\gamma} (\cdot , t)$ is
continuous from $[0,1]$ to $C^2 (\SS^2 , M)$, and (also by
\eqr{e:jinL1})
\begin{equation}
    \max_t \, \, \Area \, ( \tilde{\gamma} (\cdot , t)) < W_A + \epsilon \,
    .
\end{equation}
The maps $\tilde{\gamma} (\cdot , t)$ induce a continuous
one-parameter family of pull-back (possibly degenerate) $C^1$
metrics $g(t)$ on $\SS^2$.
  Lemma \ref{l:abe}
requires that the metrics be non-degenerate, so define perturbed
metrics $\tilde{g}(t) = g(t) + \delta \, g_0$.  For each $t$,
Lemma \ref{l:abe}
 gives $C^{1,1/2}$ conformal diffeomorphisms $h_t:
\SS^2_{g_0} \to \SS^2_{\tilde{g}(t)}$ that vary continuously in
$C^0 \cap W^{1,2} (\SS^2 , \SS^2)$.
  The continuity of $t \to \tilde{\gamma}
(\cdot , t) \circ h_t$ as a map from $[0,1]$ to $C^0 \cap W^{1,2}
(\SS^2, M)$ follows from this, the continuity of $t \to
\tilde{\gamma} (\cdot , t)$ in $C^2$, and the chain
  rule.

We will now use the conformality of the map $h_t$ to control the
energy of the composition as in proposition $5.4$ of \cite{SW}.
Namely, we have that
\begin{align}   \label{e:htg}
    \Energy \, (\tilde{\gamma}
(\cdot , t) \circ h_t) &= \Energy \,  (h_t: \SS^2_{g_0} \to
\SS^2_{{g}(t)}  ) \leq \Energy \,  (h_t: \SS^2_{g_0}
\to \SS^2_{\tilde{g}(t)}  ) \notag \\
&=  \Area \,  (\SS^2_{\tilde{g}(t)})  = \int_{\SS^2} [
\det (g_0^{-1} \, g(t)) + \delta \, \Tr  (g_0^{-1} \, g(t)) + \delta^2 ]^{1/2} \, dvol_{g_0} \\
&\leq \Area \, (\SS^2_{
 g(t)}) + 4\pi   \, [\delta^2 +  2\, \delta \, \sup_{t}\,  |g_0^{-1} \, g(t)| ]^{1/2}\, .
\notag
\end{align}
Choose $\delta > 0$ so that $4\pi \, [\delta^2 +  2\, \delta \,
\sup_{t}\,  |g_0^{-1} \, g(t)| ]^{1/2} < \epsilon$.

 We would
  be done if $\tilde{\gamma} (\cdot , t) \circ h_t$ was
   homotopic to   $\tilde{\gamma}$.  However,  the space of orientation
preserving diffeomorphisms of $\SS^2$ is homotopic to  $\RP^3$ by
Smale's theorem.   To get around this, note that $t\to ||\tilde{\gamma}
(\cdot , t)||_{C^2}$ is continuous and zero when $t=1$, thus for some
$\tau < 1$
\begin{equation}    \label{e:endpt}
   \sup_{t \geq \tau}  ||\tilde{\gamma} (\cdot ,
t)||_{C^2} \leq \frac{\epsilon}{ \sup_{t \in [0,1]}
||h_t||^2_{W^{1,2}} } \, .
\end{equation}
Consequently, if we set $\tilde{h}_t$ equal to $h_t \equiv h(t)$
on $[0,\tau]$ and equal to $h(\tau(1-t)/(1-\tau))$ on $[\tau ,
1]$, then \eqr{e:htg} and \eqr{e:endpt} imply that $
    \max_{t \in [0,1]} \, \Energy \, (\tilde{\gamma} (\cdot , t) \circ
    \tilde{h}_t ) \leq W_A  +  2\, \epsilon  \, .
$
Moreover, the map $\tilde{\gamma} (\cdot , t) \circ
    \tilde{h}_t$ is also in $\Omega$.    Finally, replacing $\tau$ by $s \tau$ and taking $s \to 0$ gives an
    explicit homotopy in $\Omega$ from $\tilde{\gamma} (\cdot , t) \circ
    \tilde{h}_t$ to $\tilde{\gamma} (\cdot , t)$.
\end{proof}

\end{document}